\documentclass[11pt,final]{amsart}
\usepackage[utf8]{inputenc}
\usepackage[T1]{fontenc}
\usepackage[francais]{babel}
\usepackage[letterpaper,centering]{geometry}
\usepackage{lmodern}
\usepackage{graphics}
\usepackage{mathrsfs}
\usepackage{amsmath,amssymb,amsfonts}
\usepackage[unicode,pdfborder={0 0 0},final]{hyperref}
\usepackage[ps,dvips,arrow,matrix,tips,line,curve]{xy}
{\setbox0\hbox{$ $}}\fontdimen16\textfont2=\fontdimen17\textfont2
\entrymodifiers={+!!<0pt,\the\fontdimen22\textfont2>}
\SelectTips{cm}{11}

\newcount\tmpstorecat
\def\csarg#1#2{\expandafter#1\csname#2\endcsname}
\def\storecat#1{\tmpstorecat\escapechar \escapechar=-1\csarg\edef{restorecat\string#1}{\catcode`\string#1=\the\catcode\expandafter`\string#1}\catcode\expandafter`\string#1=12\relax\escapechar\tmpstorecat}
\def\restorecat#1{\tmpstorecat\escapechar \escapechar=-1\csname restorecat\string#1\endcsname\escapechar\tmpstorecat}
\def\myxyin{\storecat!\storecat;}
\def\myxyout{\restorecat!\restorecat;}

\mathcode`A="7041 \mathcode`B="7042 \mathcode`C="7043 \mathcode`D="7044
\mathcode`E="7045 \mathcode`F="7046 \mathcode`G="7047 \mathcode`H="7048
\mathcode`I="7049 \mathcode`J="704A \mathcode`K="704B \mathcode`L="704C
\mathcode`M="704D \mathcode`N="704E \mathcode`O="704F \mathcode`P="7050
\mathcode`Q="7051 \mathcode`R="7052 \mathcode`S="7053 \mathcode`T="7054
\mathcode`U="7055 \mathcode`V="7056 \mathcode`W="7057 \mathcode`X="7058
\mathcode`Y="7059 \mathcode`Z="705A

\theoremstyle{plain}
\newtheorem{thm}{Théorème}[section]
\newtheorem{thmintro}{Théorème}
\newtheorem{conj}[thm]{Conjecture}
\newtheorem{lem}[thm]{Lemme}

\newtheorem{cor}[thm]{Corollaire}
\newtheorem*{cor*}{Corollaire}
\newtheorem{prop}[thm]{Proposition}
\newtheorem{rappel}[thm]{Rappel}
\theoremstyle{definition}

\newtheorem{defn}[thm]{Définition}
\newtheorem{rmk}[thm]{Remarque}

\newtheorem{exemple}[thm]{Exemple}
\numberwithin{equation}{section}

\input cyracc.def
\DeclareFontFamily{U}{russian}{}
\DeclareFontShape{U}{russian}{m}{n}
        { <5><6> wncyr5
        <7><8><9> wncyr7
        <10><10.95><12><14.4><17.28><20.74><24.88> wncyr10 }{}
\DeclareSymbolFont{Russian}{U}{russian}{m}{n}
\DeclareSymbolFontAlphabet{\mathcyr}{Russian}
\makeatletter
\let\@math@cyr\mathcyr
\renewcommand{\mathcyr}[1]{\@math@cyr{\cyracc #1}}
\makeatother
\newcommand{\Sha}{{\mathcyr{Sh}}}
\newcommand{\ab}{{\mathrm{ab}}}
\newcommand{\bark}{{\mkern1mu\overline{\mkern-1mu{}k\mkern-1mu}\mkern1mu}}
\newcommand{\aff}{\mathrm{aff}}
\newcommand{\pr}{{\mathrm{pr}}}
\newcommand{\Gammap}{\Gamma_{\mkern-2mup}}
\newcommand{\A}{{\mathbf A}}

\renewcommand{\Im}{{\mathrm{Im}}}
\renewcommand{\P}{{\mathbf P}}
\newcommand{\Q}{{\mathbf Q}}

\newcommand{\Z}{{\mathbf Z}}
\newcommand{\Rhat}{{\mkern-2mu\widehat{\mkern2muR}}}
\newcommand{\CH}{\mathrm{CH}}
\newcommand{\nr}{\mathrm{nr}}
\newcommand{\Gm}{\mathbf{G}_\mathrm{m}}
\newcommand{\et}{{\text{ét}}}
\newcommand{\Gmbark}{\mathbf{G}_{\mathrm{m},\bark}}
\newcommand{\SL}{\mathrm{SL}}
\newcommand{\Gal}{\mathrm{Gal}}
\newcommand{\Pic}{\mathrm{Pic}}
\newcommand{\Br}{\mathrm{Br}}

\newcommand{\Spec}{\mathrm{Spec}}
\newcommand{\Div}{\mathrm{Div}}
\renewcommand{\phi}{\varphi}
\renewcommand{\epsilon}{\varepsilon}
\renewcommand{\emptyset}{\varnothing}
\newcommand{\Ker}{{\mathrm{Ker}}}
\newcommand{\Coker}{{\mathrm{Coker}}}
\newcommand{\Hom}{{\mathrm{Hom}}}
\newcommand{\Homrond}{\mathscr{H}\mkern-4muom}
\newcommand{\mmu}{\boldsymbol{\mu}}
\newcommand{\chapeau}{{\rlap{\smash{\hbox{\lower4pt\hbox{\hskip1pt$\widehat{\phantom{u}}$}}}}}}
\newcommand{\zhat}{{\widehat{z}}}
\newcommand{\CHzhat}{\CH_{0\phantom{,}}^\chapeau}
\newcommand{\CHzAhat}{\CH_{0,\A}^\chapeau}
\DeclareMathOperator{\inv}{inv}
\edef\textsectiondouble{\textsection\textsection\penalty10000\hskip3.4pt}
\edef\textsection{\textsection\penalty10000\hskip3.4pt}

\makeatletter\let\@wraptoccontribs\wraptoccontribs\makeatother

\makeatletter\def\@@and{et}\makeatother

\date{26 février 2018; révisé le 23 septembre 2019}
\title[Zéro-cycles sur les espaces homogènes]{Zéro-cycles sur les espaces homogènes\\et problème de Galois inverse}

\author{Yonatan Harpaz}
\address{Institut Galil\'ee, Universit\'e Paris 13, 99~avenue Jean-Baptiste Cl\'ement, 93430 Villetaneuse, France}
\email{harpaz@math.univ-paris13.fr}

\author{Olivier Wittenberg}
\address{D\'epartement de math\'ematiques et applications, \'Ecole normale sup\'erieure, 45~rue d'Ulm, 75230 Paris Cedex 05, France}
\email{wittenberg@dma.ens.fr}

\begin{document}
\begin{abstract}
Soit~$X$ une compactification lisse d'un espace homogène d'un groupe algébrique linéaire~$G$ sur un corps de nombres~$k$.
Nous établissons la conjecture de Colliot-Thélène, Sansuc, Kato et Saito
sur l'image du groupe de Chow des zéro-cycles de~$X$ dans le produit des mêmes groupes
sur tous les complétés de~$k$.
Lorsque~$G$ est semi-simple et simplement connexe et que le stabilisateur géométrique
est fini et hyper-résoluble,
nous
montrons que les points rationnels de~$X$ sont denses dans l'ensemble de Brauer--Manin.
Pour les groupes finis hyper-résolubles,
en particulier pour les groupes finis nilpotents,
cela donne une nouvelle preuve du théorème de Shafarevich sur le problème de Galois inverse
et résout en même temps, pour ces groupes, le problème de Grunwald.
\end{abstract}

\maketitle

\section{Introduction}
\label{sec:intro}

On sait, depuis Hilbert
et Noether,
que si~$k$ est un corps de nombres et~$\Gamma$ un groupe fini,
la question de l'existence d'une extension galoisienne de~$k$
de groupe
de Galois isomorphe à~$\Gamma$ (le \og{}problème de Galois inverse\fg{})
est en réalité un problème sur l'arithmétique
des variétés algébriques quotient $\A^m_k/\Gamma$ où~$\Gamma$ agit par permutation des coordonnées
de l'espace affine~$\A^m_k$
(une fois choisis un entier~$m$ et un plongement de~$\Gamma$ dans le groupe
symétrique~$\mathfrak{S}_m$;
voir \cite[p.~xiii]{serretopics}); ou, de façon équivalente,
un problème sur l'arithmétique des espaces homogènes $\SL_n/\Gamma$
où~$\Gamma$ est vu comme
 un sous-groupe algébrique constant de~$\SL_n$
(une fois choisis un entier~$n$ et un plongement de~$\Gamma$ dans $\SL_n(k)$;
voir \cite[Corollary~3.11]{ctsanmumbai}).

Ekedahl et Colliot-Thélène ont ainsi montré que l'existence d'une extension galoisienne
de~$k$ de groupe~$\Gamma$ résulterait de la propriété dite
d'approximation très faible pour la variété $\SL_n/\Gamma$ sur~$k$
(voir \cite{ekedahl}, \cite[Theorem~3.5.7]{serretopics}, \cite[Proposition~1]{harariquelques}).
De même, le problème de Grunwald, qui consiste à déterminer dans quels cas
il existe une extension
galoisienne de~$k$ de groupe~$\Gamma$ dont les complétés en un ensemble fini de places
de~$k$ sont prescrits, se traduit en termes des points rationnels
de $\SL_n/\Gamma$ et de leur position dans l'espace des points locaux
(voir \cite{chernousov},
\cite[\textsectiondouble1.1--1.2]{harariquelques}, \cite{demarchelucchinineftin}).

La variété $\SL_n/\Gamma$ est un exemple
d'espace homogène de groupe algébrique linéaire.
De façon générale, les compactifications lisses
d'espaces homogènes de groupes algébriques linéaires sont des variétés géométriquement
unirationnelles, donc
rationnellement connexes au sens de Campana, Kollár, Miyaoka et Mori
(voir \cite[Chapter~IV]{kollarbook});
à ce titre, l'étude de leurs points rationnels est gouvernée, sur les corps
de nombres, par la conjecture suivante de
Colliot-Thélène~\cite{ctbudapest}:

\begin{conj}
\label{conj:ctrc}
Soit~$X$ une variété propre et lisse sur un corps de nombres~$k$.
Si~$X$ est rationnellement connexe,
l'ensemble~$X(k)$ est dense dans $X(\A_k)^{\Br(X)}$.
\end{conj}

On note ici $X(\A_k)^{\Br(X)}$ l'ensemble de Brauer--Manin de~$X$,
un fermé de l'espace~$X(\A_k)$ des points adéliques de~$X$,
contenant~$X(k)$,
défini par Manin~\cite{maninicm} à l'aide
du groupe de Brauer $\Br(X)=H^2_\et(X,\Gm)$
et
de la théorie du corps de classes local et global.

La conjecture~\ref{conj:ctrc} possède un analogue pour le groupe de Chow des zéro-cycles,
dû à Colliot-Thélène, Sansuc, Kato et Saito
(voir
\cite[\textsection4]{ctsansuc},
\cite[\textsection7]{katosaitocontemp}, \cite[\textsection1]{ctbordeaux},
\cite[Conjecture~1.2]{hw}),
couramment dénommé conjecture~$(E)$.
Dans le cas des variétés rationnellement connexes,
celle-ci prend la forme suivante
(voir le rappel~\ref{rappel:rce} ci-dessous):

\begin{conj}
\label{conj:e}
Soit~$X$ une variété propre et lisse sur un corps de nombres~$k$.
Notons~$\Omega_f$ l'ensemble des places finies
de~$k$
et~$\Omega_\infty$ celui de ses places infinies.
Si~$X$ est rationnellement connexe, le complexe
\begin{align}
\label{se:e}
\xymatrix@C=1em{
\CH_0(X) \ar[r] & \displaystyle\prod_{v \in \Omega_f} \CH_0(X_{k_v}) \times \prod_{v \in \Omega_\infty}
\frac{\CH_0(X_{k_v})}{N_{\bark_v/k_v}(\CH_0(X_{\bark_v}))}
\ar[r] & \Hom(\Br(X),\Q/\Z)\rlap{,}
}
\end{align}
où la seconde flèche est la somme
des accouplements
$\CH_0(X_{k_v}) \times \Br(X_{k_v}) \to \Q/\Z$
obtenus en composant les accouplements \cite[Définition~7]{maninicm} à valeurs
dans~$\Br(k_v)$ avec
l'invariant $\Br(k_v)\hookrightarrow \Q/\Z$ de la théorie du corps de classes local,
est exact.
\end{conj}

À la suite des travaux de
Legendre, Minkowski, Hasse, Châtelet,
Eichler, Landherr,
Kneser,
Harder,
Platonov,
Chernousov,
Voskresenski\u{\i} et Sansuc,
la conjecture~\ref{conj:ctrc}
fut établie par Borovoi~\cite{borovoi}
pour toute variété~$X$ birationnellement équivalente à un espace homogène
d'un groupe algébrique linéaire connexe
soumis à l'une des deux hypothèses suivantes:
soit le stabilisateur d'un point géométrique est connexe,
soit il est fini et abélien et le groupe ambiant est
semi-simple et simplement connexe.
Liang~\cite{liangarithmetic} en déduisit la validité de la conjecture~\ref{conj:e} pour les mêmes variétés.

Récemment, Demarche et Lucchini Arteche~\cite{demarchelucchinireduction}
(voir aussi \cite{lucchinireduction})
ont montré que le cas général de la conjecture~\ref{conj:ctrc} pour
une variété~$X$ birationnellement équivalente
à un espace homogène d'un groupe algébrique linéaire connexe se ramène à celui
où le groupe ambiant est~$\SL_n$ et où le stabilisateur d'un point géométrique est fini.
Ce cas particulier présente d'ailleurs un intérêt propre, puisque
la conjecture~\ref{conj:ctrc} pour une compactification lisse
de~$\SL_n/\Gamma$
implique la propriété d'approximation très faible
pour $\SL_n/\Gamma$
(voir \cite[\textsection1.3]{harariquelques}),
d'où, entre autres,
une réponse positive au problème de Galois inverse pour~$\Gamma$.

Sur tout corps de nombres~$k$,
Neukirch \cite[Corollary~2]{neukirchinv} a établi
l'approximation faible pour $\SL_n/\Gamma$
(donc la conjecture~\ref{conj:ctrc}
pour toute variété~$X$ propre, lisse et birationnellement équivalente à~$\SL_n/\Gamma$)
lorsque~$\Gamma$ est un groupe fini constant résoluble
d'ordre premier au nombre de racines de l'unité contenues dans~$k$ (ce qui exclut notamment
les groupes d'ordre pair).  Ce résultat s'étend aux groupes finis non constants
et aux espaces homogènes dépourvus de point rationnel, à condition d'imposer une restriction similaire,
forte, sur l'ordre du stabilisateur géométrique
(voir~\cite{lucchiniapproxfaible}).
Enfin, Harari~\cite{harariquelques}
a vérifié la conjecture~\ref{conj:ctrc} pour les variétés~$X$ propres, lisses et birationnellement
équivalentes à $\SL_n/\Gamma$ lorsque~$\Gamma$ est un groupe algébrique fini
produit semi-direct itéré à noyaux abéliens.

Le présent article introduit une méthode nouvelle pour les
espaces homogènes de groupes algébriques linéaires semi-simples simplement
connexes à stabilisateur géométrique fini et
l'applique à l'étude des conjectures~\ref{conj:ctrc} et~\ref{conj:e}
dans le cadre des espaces homogènes de
groupes linéaires.
Les deux théorèmes
principaux qu'elle nous permet d'obtenir sont les suivants.  Le premier
porte sur les zéro-cycles et résout, dans ce cadre,
la conjecture~\ref{conj:e} en toute généralité,
c'est-à-dire sans aucune hypothèse sur le stabilisateur.

\begin{thmintro}
\label{th but}
La conjecture~\ref{conj:e} vaut,
sur tout corps de nombres,
pour toute variété propre et lisse
birationnellement équivalente à un espace homogène d'un groupe
algébrique linéaire.
\end{thmintro}

Notant $\Br_\nr(V)$ le groupe de Brauer d'une compactification lisse
d'un espace homogène~$V$, le théorème~\ref{th but}
a pour corollaire immédiat,
compte tenu du lemme de déplacement pour les zéro-cycles \cite[p.~599]{ctfinitude}, l'énoncé d'existence
suivant:

\begin{cor*}
Soit~$k$ un corps de nombres.
Soit~$V$ un espace homogène d'un groupe algébrique linéaire connexe sur~$k$.
Si $V(\A_k)^{\Br_\nr(V)}\neq\emptyset$, il existe un zéro-cycle de degré~$1$ sur~$V$.
\end{cor*}

Le second théorème porte sur les points rationnels.
Contrairement au théorème~\ref{th but}, qui repose sur~\cite{hw},
le théorème~\ref{th points rat hr} ne s'appuie que sur des résultats
connus depuis une vingtaine d'années
(notamment dus à Harari~\cite{harariduke}, \cite{hararifleches}).

\begin{thmintro}
\label{th points rat hr}
Soit~$X$ une variété propre, lisse et géométriquement irréductible sur
un corps de nombres~$k$.
Si~$X$ est
birationnellement équivalente à un espace homogène d'un groupe
algébrique linéaire semi-simple simplement connexe, à stabilisateur géométrique
fini hyper-résoluble (en tant que groupe fini muni d'une action extérieure
du groupe de Galois absolu de~$k$; voir la définition~\ref{def:hyperresoluble} ci-dessous),
alors $X(k)$ est dense dans $X(\A_k)^{\Br(X)}$.

En particulier, la conjecture~\ref{conj:ctrc} vaut pour les variétés
propres et lisses birationnellement
équivalentes à $\SL_n/\Gamma$ avec~$\Gamma$ fini constant hyper-résoluble (par exemple, nilpotent).
\end{thmintro}

Le théorème~\ref{th points rat hr} implique une réponse positive
au problème de Galois inverse pour les groupes finis hyper-résolubles, en particulier pour les groupes
finis nilpotents, sur tout corps de nombres.  Une réponse positive
au problème de Galois inverse
avait déjà été apportée
par Shafarevich pour ces groupes, et plus généralement pour les groupes finis résolubles
(voir \cite[Chapter~IX, \textsection6]{neukirchschmidtwingberg} et les références qui y sont données).
Cependant, le théorème~\ref{th points rat hr} donne des informations significativement plus précises que la
seule existence d'une extension finie galoisienne de~$k$ de groupe de Galois donné.
Ainsi, par exemple, en le combinant avec les résultats récents de Lucchini
Arteche~\cite[\textsection 6]{lucchiniunramifiedbrauer} sur le groupe de
Brauer non ramifié des espaces homogènes, on obtient,
compte tenu de \cite[Proposition~2.4]{demarchelucchinineftin},
une solution au \og{}problème d'approximation modéré\fg{}
de \emph{op.\ cit.}, \textsection1.2,
pour les groupes hyper-résolubles.
Dans le cas des stabilisateurs constants, cela se traduit par
l'énoncé suivant, auparavant inconnu même pour les groupes finis nilpotents:

\begin{cor*}
Soit~$\Gamma$ un groupe fini hyper-résoluble.
Soit~$k$ un corps de nombres.
Soit~$S$ un ensemble fini de places de~$k$
ne contenant aucune place finie divisant l'ordre de~$\Gamma$.
Pour chaque $v \in S$, fixons
une extension galoisienne~$K_v/k_v$
dont le groupe de Galois se plonge dans~$\Gamma$.
Il existe alors une extension finie galoisienne~$K/k$ de groupe de Galois~$\Gamma$ telle que pour chaque $v \in S$,
l'extension de~$k_v$ obtenue en complétant~$K$ en une place divisant~$v$
soit isomorphe à $K_v/k_v$.
\end{cor*}

Ce corollaire, qui
répond par l'affirmative,
pour les groupes
finis hyper-résolubles,
au problème de Grunwald
hors des places divisant l'ordre de~$\Gamma$
(voir \cite{demarchelucchinineftin}),
ne semble pas accessible aux méthodes de Shafarevich.

Comme on le sait depuis Wang~\cite{wang},
l'hypothèse portant sur~$S$ est essentielle à sa validité,
même lorsque $\Gamma=\Z/8\Z$, $k=\Q$, $S=\{2\}$.  D'autres exemples
montrant que cette hypothèse est cruciale
sont donnés dans~\cite[Theorem~1.2 et \textsection5]{demarchelucchinineftin}.

Des informations plus fines sur le groupe de Brauer non ramifié d'un espace homogène donné
permettent de déduire du théorème~\ref{th points rat hr} des corollaires plus
précis que celui que l'on vient d'énoncer.
Par exemple, si $\Gamma=Q_{2^m}$ est le groupe quaternionique d'ordre~$2^m$ pour un entier $m\geq 1$,
vu comme sous-groupe algébrique constant de~$\SL_n$
une fois un plongement dans $\SL_n(k)$ choisi,
Demarche~\cite[Corollaire~3, Remarque~7]{demarchebrnr} a démontré que le groupe de Brauer
non ramifié de $\SL_n/\Gamma$ est réduit à~$\Br(k)$.
Le groupe~$\Gamma$ est un $2$\nobreakdash-groupe,
donc est nilpotent, donc est hyper-résoluble;
il résulte donc du théorème~\ref{th points rat hr}
que la variété $\SL_n/\Gamma$ vérifie l'approximation faible.
Autrement dit, le corollaire au théorème~\ref{th points rat hr}
reste vrai pour $\Gamma=Q_{2^m}$ si l'on prend pour~$S$ un ensemble fini \emph{quelconque} de places de~$k$.
Un tel énoncé était précédemment connu pour $m\leq 4$ seulement
(voir \cite[Théorème~7, Remarque~13]{demarchebrnr}).

\medskip
Les résultats de Borovoi~\cite{borovoi} et Liang~\cite{liangarithmetic} mentionnés plus haut
ne prennent en compte que le sous-groupe \emph{algébrique} du groupe de Brauer non ramifié
de l'espace homogène en question, c'est-à-dire le sous-groupe constitué des classes annulées par une
extension des scalaires.  Sous les hypothèses de leurs théorèmes,
tout le groupe de Brauer non ramifié est en fait algébrique, comme l'ont
montré Bogomolov,
Borovoi, Demarche et Harari
(voir~\cite{bogomolov}, \cite{borovoidemarcheharari}
et les références données dans
\cite[\textsection3.2.1]{wslc}).
Les théorèmes~\ref{th but} et~\ref{th points rat hr},
en revanche,
prennent
en compte les classes \emph{transcendantes} (\emph{i.e.}\ non algébriques) du groupe de Brauer non ramifié.
En effet,
dans la situation du théorème~\ref{th points rat hr},
Saltman et Bogomolov
ont établi dans les années~1980 que le groupe de Brauer de~$X$ peut contenir des éléments transcendants
(voir~\cite{saltmannoether},
\cite{bogomolovbrauer};
en particulier~$X$ n'est pas géométriquement rationnelle en général);
Demarche, Lucchini Arteche et
Neftin~\cite{demarchelucchinineftin} ont prouvé que ceux-ci jouent un rôle
dans l'obstruction de Brauer--Manin pour de tels espaces homogènes.

\medskip
Les démonstrations des théorèmes~\ref{th but} et~\ref{th points rat hr} reposent
sur une observation
simple concernant la géométrie des espaces homogènes de~$\SL_n$
à stabilisateur
géométrique~$\Gamma$ fini, à savoir:
\emph{les torseurs universels de leurs compactifications lisses} (au sens de la théorie
de la descente de Colliot-Thélène et Sansuc~\cite{ctsandescent2})
\emph{contiennent chacun un ouvert fibré au-dessus d'un tore quasi-trivial
en espaces homogènes de~$\SL_n$ à stabilisateur géométrique le groupe dérivé~$\Gamma'$}.

Lorsque~$\Gamma$ est résoluble, on voit apparaître là une structure géométrique inductive.
C'est cette structure que nous exploitons dans l'article, à l'aide de deux ingrédients:
la méthode des fibrations, qui vise à déduire un énoncé arithmétique, pour l'espace total d'une
fibration, du même énoncé pour ses fibres et pour sa base,
et la méthode de la descente, qui vise à déduire un énoncé arithmétique, pour une variété donnée, du même
énoncé pour ses torseurs universels.

Dans le cadre des zéro-cycles, la méthode des fibrations est directement applicable sous la forme
qui lui est donnée dans~\cite{hw}.  Dans le cadre des points rationnels, les résultats de~\cite{hw}
ne permettent d'appliquer la méthode des fibrations que de façon conditionnelle;
nous établissons dans le présent article
des énoncés inconditionnels qui suffisent à traiter le cas des stabilisateurs hyper-résolubles
et dont les preuves reposent notamment
sur les travaux de Harari~\cite{harariduke}, \cite{hararifleches}.
Il nous est également nécessaire d'étendre la théorie de la descente de Colliot-Thélène et Sansuc
afin de montrer que
 la conjecture~\ref{conj:ctrc}
pour une variété propre, lisse et rationnellement connexe
se déduit de la même conjecture
pour les compactifications lisses de ses torseurs universels.
Cet énoncé général ne résulte de~\cite{ctsandescent2} que dans le cas où le groupe de Brauer de la variété
considérée est constitué uniquement de classes algébriques
(ce qui, d'après Saltman et Bogomolov, n'est pas le cas ici).

Il est à noter que même si
l'espace homogène auquel on s'intéresse est $\SL_n/\Gamma$ pour un groupe~$\Gamma$ fini constant,
les espaces homogènes qui apparaissent dans la fibration que l'on vient d'évoquer
ne sont que des formes
de $\SL_n/\Gamma'$: ils ne possèdent pas nécessairement de point rationnel (et quand ils en possèdent,
le stabilisateur n'est pas nécessairement un groupe constant).
Même pour les seules applications au problème de Galois inverse,
il est donc essentiel à la stratégie de la démonstration
que le théorème~\ref{th points rat hr} s'applique aussi aux espaces homogènes dépourvus de point rationnel.

\medskip
Le texte est organisé comme suit.
Le~\textsection\ref{sec:descente} est consacré à la théorie de la descente sous un tore pour les variétés rationnellement
connexes.
Le~\textsection\ref{sec:duntorseuralautre} compare la géométrie des torseurs universels d'une variété avec ceux
d'un ouvert dense de la même variété et montre que les premiers contiennent des ouverts fibrés en les seconds
au-dessus d'un tore quasi-trivial.
Au~\textsection\ref{sec:fibrations tore qtriv}, nous discutons la méthode des fibrations pour les variétés
fibrées en variétés lisses au-dessus d'un tore quasi-trivial.
Nous nous tournons, au~\textsection\ref{sec:pi1eh}, vers les revêtements étales des espaces homogènes
de groupes algébriques linéaires connexes semi-simples et simplement connexes et identifions leurs
torseurs universels.
Enfin, les~\textsectiondouble\ref{sec:points rat}
et~\ref{sec:zerocycles}
démontrent, respectivement, les théorèmes~\ref{th points rat hr} et~\ref{th but}.

\bigskip
\emph{Remerciements.}
Nous sommes reconnaissants à Cyril Demarche et Giancarlo Lucchini Arteche de nous avoir
transmis une version préliminaire de leur article~\cite{demarchelucchinireduction},
dont dépend le~\textsection\ref{subsec:reductionstabfini} du présent article,
et aux rapporteurs de leur relecture attentive.
Le premier auteur remercie l'Institut des Hautes Études Scientifiques
pour son hospitalité pendant l'élaboration de ce travail.

\bigskip
\emph{Notations et terminologie.}
Dans tout l'article, on désigne par~$\bark$ une clôture algébrique d'un corps~$k$ de caractéristique nulle.
Les isomorphismes sont notés $\simeq$, les isomorphismes canoniques~$=$.
Tous les groupes de cohomologie sont des groupes de cohomologie galoisienne
ou étale.
Lorsque~$X$ est une variété sur~$k$ et $k'/k$ est une extension, nous noterons indifféremment~$X_{k'}$
ou~$X\otimes_k k'$ la variété sur~$k'$ obtenue par extension des scalaires,
en privilégiant la seconde notation lorsqu'il est utile de préciser le corps de base.

Une \emph{action extérieure} d'un groupe profini~$\Gamma$ sur un groupe fini~$H$ est un morphisme
de groupes continu $\Gamma \to \mathrm{Out}(H)$, où
le groupe~$\mathrm{Out}(H)$ des automorphismes extérieurs
de~$H$ est muni de la topologie discrète.

Un \emph{tore} sur~$k$ est un groupe algébrique~$T$ tel que
$T_\bark \simeq \Gmbark\times\dots\times\Gmbark$.
Il est \emph{quasi-trivial}
s'il existe une $k$\nobreakdash-algèbre étale~$E$,
c'est-à-dire une $k$\nobreakdash-algèbre
isomorphe à un produit fini d'extensions finies séparables de~$k$,
telle que $T\simeq R_{E/k}\Gm$,
où~$R_{E/k}$ désigne la restriction des scalaires à la Weil.
Un \emph{groupe de type multiplicatif} sur~$k$ (toujours supposé de caractéristique nulle)
est un groupe algébrique commutatif extension d'un groupe
algébrique fini par un tore.
Un \emph{espace homogène} d'un groupe algébrique~$G$ sur~$k$ est une variété non vide~$V$ sur~$k$
munie d'une action de~$G$ telle que~$G(\bark)$ agisse transitivement sur~$V(\bark)$.

Nous employons une notation multiplicative pour le groupe $T(\bark)$ (dont le neutre est~$1$),
additive pour le groupe $Z^1(k,T)$
des $1$\nobreakdash-cocycles continus de $\Gal(\bark/k)$ à valeurs dans
le module galoisien discret~$T(\bark)$ (le cocycle neutre est donc~$0$).
On écrira $[\sigma] \in H^1(k,T)$ pour désigner un élément de $H^1(k,T)$
et en choisir un représentant $\sigma \in Z^1(k,T)$.
Si $f:Y \to X$ est un torseur sous~$T$, on notera
$f^\sigma:Y^\sigma \to X$ le torseur sous~$T$ tordu de~$f$ par $\sigma \in Z^1(k,T)$.
Autrement dit,
si~$Z$ désigne le torseur sous~$T$ sur~$\Spec(k)$ déterminé par~$\sigma$,
alors $Y^\sigma = Y\times_k^T Z$ est le produit contracté de~$Y$ et de~$Z$ sous l'action de~$T$.
On renvoie à
\cite[Chapter~2]{skobook} pour toutes ces notions.

Une partie~$H$ d'une variété irréductible~$X$ est un sous-ensemble \emph{hilbertien}
s'il existe un ouvert dense $X^0 \subset X$, un entier $n\geq 1$
et des $X^0$\nobreakdash-schémas finis étales irréductibles $W_1,\dots,W_n$
tels que~$H$ soit l'ensemble des points de~$X^0$
au-dessus desquels la fibre de~$W_i$ est irréductible pour tout $i \in \{1,\dots,n\}$.

Si~$V$ est une variété lisse et séparée sur~$k$, une \emph{compactification lisse}
de~$V$ est une variété~$X$ propre et lisse sur~$k$ contenant~$V$ comme ouvert dense.
À plusieurs reprises
nous ferons usage du fait suivant:
si $f:W \to V$ est un morphisme entre variétés lisses et séparées sur~$k$,
alors~$V$ et~$W$ admettent des compactifications lisses et
pour
toute compactification lisse~$X$ de~$V$,
il existe une compactification lisse~$Y$
de~$W$ telle que l'application rationnelle $Y \dashrightarrow X$ définie par~$f$
soit un morphisme
(voir~\cite{nagata}, \cite{hironaka}).

Lorsque~$X$ est une variété irréductible et lisse sur~$k$, le \emph{groupe de Brauer non ramifié}
$\Br_\nr(X)$ est le groupe de Brauer de toute variété propre et lisse birationnellement équivalente à~$X$.
On
pose
$\Br_1(X)=\Ker\big(\Br(X) \to \Br(X_{\bark})\big)$
et $\Br_0(X)=\Im\big(\Br(k)\to\Br(X)\big)$
et l'on note $\bark[X]^*$ le groupe des fonctions inversibles sur~$X_\bark$.
Si~$X$ est propre, on dit que~$X$ est \emph{rationnellement connexe} si $X_{\bark}$ est rationnellement connexe
au sens de \cite[Chapter~IV, Definition~3.2.2]{kollarbook}; c'est notamment le cas si~$X_{\bark}$
est unirationnelle.

Lorsque~$k$ est un corps de nombres, on note $\Omega$ l'ensemble de ses places.
Pour tout entier~$i$ et tout $\Gal(\bark/k)$\nobreakdash-module discret~$M$
(ou tout groupe algébrique commutatif~$M$ sur~$k$),
on note $\Sha^i(k,M)=\Ker\big(H^i(k,M)\to\prod_{v\in\Omega}H^i(k_v,M)\big)$.
Une variété~$X$ sur~$k$ vérifie \emph{l'approximation faible}
si~$X(k)$ est dense dans $\prod_{v\in\Omega}X(k_v)$ pour la topologie produit.
(On ne suppose pas que~$X(k)$ est non vide.)
Lorsque~$X$ est propre et lisse,
l'application du foncteur
$M\mapsto \smash[t]{\mkern1mu\widehat{\mkern-1muM\mkern-1mu}\mkern1mu}=\varprojlim_{n\geq 1}M/nM$
au complexe~\eqref{se:e} fournit, compte tenu que $\Br(X)$ est un groupe de torsion, un complexe
\begin{align}
\label{se:egeneral}
\xymatrix@C=1em{
\CHzhat(X) \ar[r] & \CHzAhat(X) \ar[r] & \Hom(\Br(X),\Q/\Z)\rlap{,}
}
\end{align}
où l'on a posé
$\CH_{0,\A}(X)=\prod_{v \in \Omega_f} \CH_0(X_{k_v}) \times \prod_{v \in \Omega_\infty}
{\CH_0(X_{k_v})}/{N_{\bark_v/k_v}(\CH_0(X_{\bark_v}))}$.
On dit que~$X$ \emph{vérifie la conjecture~$(E)$} si le complexe~\eqref{se:egeneral} est exact.

\begin{rappel}
\label{rappel:rce}
Si~$X$ est une variété propre, lisse et rationnellement connexe, sur un corps de nombres~$k$,
l'exactitude du complexe~\eqref{se:e} équivaut à celle du complexe~\eqref{se:egeneral}.
\end{rappel}

\begin{proof}
Supposons~$X$ non vide, notons $\epsilon:X\to \Spec(k)$ le morphisme structural
et considérons le diagramme commutatif
\begin{align}
\begin{aligned}
\label{diag:rappel}
\xymatrix@R=3ex@C=1.196em{
0 \ar[r] & \CH_0(X) \ar[d] \ar[r] & \CHzhat(X) \ar[d] \ar[r]^(.225){\epsilon_*} & \Coker\big(\CH_0(\Spec(k))\to \CHzhat(\Spec(k))\big) \ar[r]
\ar[d] & 0 \\
0 \ar[r] & \CH_{0,\A}(X) \ar[r] & \CHzAhat(X) \ar[r]^(.225){\epsilon_*} & \Coker\big(\CH_{0,\A}(\Spec(k))\to\CHzAhat(\Spec(k))\big) \ar[r] & 0\rlap{.}
}
\end{aligned}
\end{align}

Vérifions l'exactitude des lignes.
Les homomorphismes
$\epsilon_*:\CH_0(X) \to \CH_0(\Spec(k))$ et $\epsilon_*:\CH_{0,\A}(X) \to \CH_{0,\A}(\Spec(k))$
ont leur noyau et leur conoyau d'exposant fini (voir \cite[Proposition~11]{ctfinitude}
et \cite[Lemme~2.4]{wittdmj}).
Comme les groupes
$\CH_0(\Spec(k))$ et $\CH_{0,\A}(\Spec(k))$ n'ont pas d'élément infiniment divisible non nul
et comme leur torsion est finie et annulée par~$2$,
l'exactitude à gauche et au milieu des lignes de~\eqref{diag:rappel}
s'ensuit (voir \cite[Lemme~1.11 et Lemme~1.12]{wittdmj}).
Comme $\epsilon_*:\CHzhat(X)\to \CHzhat(\Spec(k))$
et $\epsilon_*:\CHzAhat(X) \to \CHzAhat(\Spec(k))$ ont aussi leur conoyau d'exposant fini
(même preuve que \cite[Lemme~2.4]{wittdmj}) et comme les groupes de droite
de~\eqref{diag:rappel} sont divisibles (voir \cite[Lemme~1.10]{wittdmj}),
les flèches horizontales de droite de~\eqref{diag:rappel} ont leur conoyau à la fois
divisible et d'exposant fini, donc nul.  Ainsi les lignes de~\eqref{diag:rappel} sont-elles bien exactes.

La flèche verticale de droite du diagramme~\eqref{diag:rappel} est visiblement injective
et son image contient la classe de $\epsilon_*\zhat_\A$
pour tout $\zhat_\A \in \CHzAhat(X)$ orthogonal à~$\Br(X)$
(voir \cite[Remarque~1.1~(i)]{wittdmj}).  Ces remarques, jointes
à l'exactitude des lignes, entraînent la conclusion voulue par une chasse au diagramme.
\end{proof}

\section{Descente sous un tore et groupe de Brauer transcendant}
\label{sec:descente}

\subsection{Énoncés}

Soit~$X$ une variété propre, lisse et rationnellement connexe, sur un corps de nombres~$k$.
La théorie de la descente de Colliot-Thélène et Sansuc \cite{ctsandescent2}
affirme que
pour tout torseur $f:Y\to X$ sous un tore~$T$ sur~$k$,
on a une inclusion
\begin{align}
\label{eq:inclusion descente classique}
X(\A_k)^{\Br_1(X)} \subset \bigcup_{[\sigma]\in H^1(k,T)} f^\sigma(Y^\sigma(\A_k))
\end{align}
de sous-ensembles de $X(\A_k)$.
Si~$f$ est un torseur universel,
le groupe de Brauer non ramifié
algébrique de~$Y^\sigma$ est réduit aux constantes pour tout~$\sigma$
(\emph{op.\ cit.}, Théorème~2.1.2).
Si de plus~$X$ est géométriquement rationnelle,
on a $\Br(X)=\Br_1(X)$
et $\Br_\nr(Y^\sigma)=\Br_0(Y^\sigma)$ pour tout~$\sigma$,
si bien que
l'inclusion~\eqref{eq:inclusion descente classique}
ramène la conjecture~\ref{conj:ctrc}
pour~$X$ à la même conjecture pour les compactifications lisses des variétés~$Y^\sigma$.
Si en revanche~$X$ est seulement supposée rationnellement connexe,
le groupe de Brauer non ramifié de~$Y^\sigma$ peut contenir des classes transcendantes,
auquel cas l'inclusion~\eqref{eq:inclusion descente classique} ne permet plus une telle réduction.
Le but du \textsection\ref{sec:descente} est de démontrer le théorème
suivant, qui améliore~\eqref{eq:inclusion descente classique} et rend possible
la descente sur les variétés rationnellement connexes.

\begin{thm}
\label{th descente}
Soit $X$ une variété lisse et géométriquement irréductible
sur un corps de nombres~$k$.
Soit~$T$ un tore sur~$k$.
Soit $f:Y\to X$ un torseur sous~$T$.
Notons $A \subset \Br(X)$ l'image réciproque de
 $\Br_{\nr}(Y) \subset \Br(Y)$
par $f^*:\Br(X)\to\Br(Y)$.
Alors
\begin{align}
\label{eq:inclusion du th descente}
X(\A_k)^A
\subset
\bigcup_{[\sigma] \in H^1(k,T)} f^\sigma\Big(Y^\sigma(\A_k)^{\Br_{\nr}(Y^\sigma)}\Big)\rlap{\text{.}}
\end{align}
\end{thm}

Bien entendu, si~$X$ est propre, alors $A=\Br(X)=\Br_{\nr}(X)$.

Le théorème~\ref{th descente}
peut aussi se déduire de travaux récents de Cao \cite[Theorem~5.9]{caoapproxforte}.
La démonstration que nous donnons ci-dessous fut obtenue indépendamment.
D'autre part,
on trouvera
dans \cite[Theorem~1.2]{caodemarchexu}
et \cite[Theorem~1.7]{weiopendescent}
des prédécesseurs du théorème~\ref{th descente}, dans lesquels $\Br_\nr(Y^\sigma)$
est remplacé par $\Br_\nr(Y^\sigma) \cap \Br_1(Y^\sigma)$.

\begin{cor}
\label{cor:descente}
Soit $X$ une variété propre, lisse et rationnellement connexe,
sur un corps de nombres~$k$.
Soit~$T$ un tore sur~$k$.
Soit $V \subset X$ un ouvert dense.
Soit $f:W\to V$ un torseur sous~$T$.
Alors
$X(\A_k)^{\Br(X)}$
est contenu dans l'adhérence de
\begin{align*}
\bigcup_{[\sigma] \in H^1(k,T)} f^\sigma\Big(W^\sigma(\A_k)^{\Br_{\nr}(W^\sigma)}\Big)
\end{align*}
dans $X(\A_k)$.
\end{cor}

\begin{proof}[Démonstration du corollaire~\ref{cor:descente}]
Appliquons le théorème~\ref{th descente} à~$V$.
Soit $A \subset \Br(V)$ l'image réciproque de $\Br_\nr(W)$ par $f^*:\Br(V)\to\Br(W)$.
Choisissons une compactification lisse~$Y$ de~$W$ telle que~$f$ s'étende
en un morphisme $g:Y\to X$.

Le morphisme~$g$ est plat au-dessus du complémentaire d'un
fermé $F \subset X$ de codimension~$\geq 2$.
Posons $X^0=X \setminus F$
et $Y^0=Y \setminus g^{-1}(F)$
et notons $g^0:Y^0\to X^0$ la restriction de~$g$.
Comme $H^1(\bark(V),T)\simeq H^1(\bark(V),\Gm^r)=0$,
le morphisme $g^0\otimes_k\bark$ admet une section rationnelle.
Par conséquent, les fibres de~$g^0$ au-dessus des points de codimension~$1$
de~$X$ contiennent toutes une composante irréductible de multiplicité~$1$.
Appliquant (la preuve de) \cite[Lemma~3.1]{ctskodescent} à~$g^0$,
on en déduit que le groupe $A/\Br(X^0)$ est fini.
(Ce lemme affirme la finitude
de l'image de ce groupe dans $\Br(Y^0)/g^{0*}\Br(X^0)$ mais la preuve
donnée dans \emph{loc.\ cit.}\ est en réalité une preuve de la finitude de $A/\Br(X^0)$.)

Comme~$F$ est de codimension~$\geq 2$, l'inclusion $\Br(X) \subset \Br(X^0)$
est une égalité
(voir \cite[Corollaire~6.2]{brauerIII}).
Comme~$X$ est rationnellement connexe, le groupe $\Br(X)/\Br_0(X)$ est fini
(voir \cite[Lemma~1.3]{ctskogoodreduction}).
Le groupe $A/\Br_0(X)$ est donc fini.
Il en résulte, par le ``lemme formel'' de Harari,
que $X(\A_k)^{\Br(X)}$ est contenu
dans l'adhérence de $V(\A_k)^A$
dans~$X(\A_k)$
(voir \cite[Proposition~1.1]{ctskodescent}).
On conclut avec l'inclusion~\eqref{eq:inclusion du th descente} pour~$V$.
\end{proof}

\subsection{Remarques préliminaires}
\label{subsec:lemmes prelim}

Dans le~\textsection\ref{subsec:lemmes prelim}, le corps~$k$ est un corps de caractéristique nulle
quelconque et~$X$, $T$, $Y$, $f$
sont comme dans le théorème~\ref{th descente}.
La notation suivante sera utile:
si $M \to X$ est un morphisme de variétés,
on désignera par $\Br_1(M/X)$ le sous-groupe de~$\Br(M)$ constitué des classes
dont l'image réciproque dans $\Br(M_{\bark})$ provient de $\Br(X_{\bark})$.
Remarquons que $\Br_\nr(Y^\sigma) \subset \Br_1(Y^\sigma/X)$
puisque $Y^\sigma_\bark$ et $X^{\vphantom{0}}_\bark \times \P^r_{\bark}$
sont birationnellement équivalents au-dessus de~$X_\bark$ pour $r=\dim(T)$.

Soit $\sigma \in Z^1(k,T)$.
Comme $f^\sigma:Y^\sigma \to X$ est un torseur sous~$T$, on a une suite exacte
\begin{align}
\label{se:ysigma}
1\to\Gm\to f^\sigma_*\Gm \to \widehat T \to 1
\end{align}
de faisceaux étales sur~$X$,
où~$\widehat T$ désigne le groupe des caractères du tore~$T$
(voir \cite[Proposition~1.4.2]{ctsandescent2}).
D'autre part,
on a $R^1 f^\sigma_*\Gm=0$
puisque $Y^\sigma \simeq X \times \Gm^r$
localement sur~$X$ pour la topologie étale,
et l'on a aussi $H^1(X_{\bark},\widehat T)\simeq H^1(X_{\bark},\Z^r)=0$.
Il s'ensuit, au vu des suites spectrales de Leray pour~$f^\sigma$ et de Hochschild--Serre pour~$X$,
que
$H^2(X,f^\sigma_*\Gm)=\Ker\mkern1mu\big(\Br(Y^\sigma)\to H^0(X,R^2f^\sigma_*\Gm)\big)$,
$H^2(X_{\bark},f^\sigma_*\Gm)=\Ker\mkern1mu\big(\Br(Y_{\bark}^\sigma)\to H^0(X_{\bark},R^2f^\sigma_*\Gm)\big)$
et
$\Ker\mkern1mu\big(H^2(X,\widehat T)\to H^2(X_{\bark},\widehat T)\big)=H^2(k,\widehat T)$.
Appliquant les foncteurs $H^2(X,-)$,
 $H^3(X,-)$ et $H^2(X_{\bark},-)$ à~\eqref{se:ysigma}, on déduit alors,
compte tenu que
$H^0(X,R^2f^\sigma_*\Gm)$ s'injecte dans $H^0(X_\bark,R^2f^\sigma_*\Gm)$,
d'abord une identification entre le groupe $\Br_1(Y^\sigma/X)$
et le noyau de l'application composée
$H^2(X,f^\sigma_*\Gm) \to H^2(X_\bark,f^\sigma_*\Gm) \to H^2(X_\bark,\widehat T)$,
puis
une suite exacte
\begin{align}
\label{se:remarques prelim}
\xymatrix{
\Br(X) \ar[r]^(.4){(f^\sigma)^*} & \Br_1(Y^\sigma/X) \ar[r]^(.53){\phi^\sigma} & H^2(k,\widehat T) \ar[r]^(.46)\delta & H^3(X,\Gm)\rlap{.}
}
\end{align}

\begin{prop}
\label{prop:descente prelim}
Notant comme ci-dessus $\sigma$ un élément de $Z^1(k,T)$, on a:
\begin{enumerate}
\setlength\itemsep{0em}
\vspace*{-.1em}
\item[(i)]
Si $H^3(k,\Gm)=0$, l'application~$\delta$ ne dépend pas de~$\sigma$.
\item[(ii)]
Si $H^3(k,\Gm)=0$, l'image de $\Br_\nr(Y^\sigma)$ par~$\phi^{\sigma}$
ne dépend pas de~$\sigma$.
\item[(iii)]
L'image réciproque de
$\Br_{\nr}(Y^\sigma)$
par $(f^\sigma)^*$
ne dépend pas de~$\sigma$, \emph{i.e.}\ est égale à~$A$.
\item[(iv)]
Pour $y \in Y^\sigma(k)$, $t \in T(k)$ et $\alpha^\sigma \in \Br_1(Y^\sigma/X)$,
on a $\alpha^\sigma(t\cdot y)=\alpha^\sigma(y) + (\phi^\sigma(\alpha^\sigma)\smile t)$.
\end{enumerate}
\end{prop}

\begin{proof}
Pour $x \in H^2(k,\widehat T)$,
on a $\delta(x)=p^*x \smile [f^\sigma]$ à un signe indépendant de~$\sigma$ près,
si $p:X\to \Spec(k)$ désigne le morphisme structural
et $[f^\sigma] \in H^1(X,T)$ la classe de~$f^\sigma$
(voir \cite[Proposition~1.4.3]{ctsandescent2}).
Comme $[f^\sigma]=[f]+p^*[\sigma]$
et comme $x \smile [\sigma] \in H^3(k,\Gm)$,
l'assertion~(i) s'ensuit.

L'assertion~(iv) résulte du lemme ci-dessous appliqué à $Z=(f^\sigma)^{-1}(f^\sigma(y))$.

\begin{lem}
\label{lem:ancien sous-lemme}
Soit~$Z$ un torseur sous~$T$, sur~$k$.
Soient $z \in Z(k)$, $t \in T(k)$,
$\alpha \in \Br_1(Z)$.
On a $\alpha(t\cdot z)=\alpha(z)+ (\phi(\alpha)\smile t)$,
où $\phi:\Br_1(Z) \to H^2(k,\widehat T)$ désigne l'application
induite par la suite exacte
$1\to \bark{}^*\to \bark[Z]^*\to \widehat T \to 1$
et par l'isomorphisme $\Br_1(Z)=H^2(k,\bark[Z]^*)$ issu de la suite spectrale de Hochschild--Serre.
De plus~$\phi$ est surjective si $H^3(k,\Gm)=0$.
\end{lem}

\begin{proof}
L'application $\Br_1(Z) \to \Br(k)$, $\alpha \mapsto \alpha(t\cdot z)-\alpha(z)$
est celle induite,
via l'isomorphisme $\Br_1(Z)=H^2(k,\bark[Z]^*)$, par 
l'application $\bark[Z]^*\to \bark{}^*$, $a \mapsto a(t\cdot z)/a(z)$.
Cette dernière se factorise par~$\widehat T$.
Par le lemme de Rosenlicht
\cite[Proposition~3]{rosenlicht},
la flèche $\widehat T \to \bark{}^*$ qui en résulte
est $t \in T(k)=\Hom(\widehat T,\Gm)$;
d'où $\alpha(t\cdot z)=\alpha(z)+(\phi(\alpha)\smile t)$.
\end{proof}

Pour établir~(ii), fixons $\sigma,\sigma' \in Z^1(k,T)$
et montrons que
$\phi^\sigma(\Br_\nr(Y^\sigma))\subset \phi^{\sigma'}(\Br_\nr(Y^{\sigma'}))$.
Quitte à remplacer~$f$ par~$f^{-\sigma'}$ et~$\sigma$ par~$\sigma-\sigma'$,
on peut supposer que~$\sigma'=0$.

Notons~$Z$ le torseur sous~$T$, sur~$k$, déterminé par~$\sigma$.
Notons $\pi:Y \times_k Z \to Y^\sigma$
l'application canonique.
Notons $\pr_1:Y\times_k Z \to Y$, $\pr_2:Y\times_k Z \to Z$,
$g:X \times_k Z \to X$ les projections
et posons $h =f \circ \pr_1 : Y\times_k Z \to X$.

Soit $\alpha^\sigma \in \Br_\nr(Y^\sigma)$.
Comme $H^3(k,\Gm)=0$, l'application~$\phi$ du lemme~\ref{lem:ancien sous-lemme} est surjective:
il existe $\alpha \in \Br_1(Z)$ tel que $\phi(\alpha)=\phi^\sigma(\alpha^\sigma)$.
D'après la proposition~\ref{prop:descente prelim}~(i),
les applications $\phi^0$ et $\phi^\sigma$ ont la même image:
il existe donc $\alpha^0 \in \Br_1(Y/X)$
tel que $\phi^0(\alpha^0)=\phi^\sigma(\alpha^\sigma)$.
Notons $\alpha^0\boxplus \alpha=\pr_1^*\alpha^0 + \pr_2^*\alpha \in \Br_1(Y\times_k Z/X)$.
Voyons respectivement
$\alpha^\sigma$, $\alpha^0$, $\alpha$,
$\alpha^0\boxplus\alpha$
 comme des éléments de
$H^2(X,f^{\sigma}_*\Gm)$,
$H^2(X,f_*\Gm)$,
$H^2(X,g_*\Gm)$,
$H^2(X,h_*\Gm)$.
Le diagramme commutatif à lignes exactes
\begin{align*}
\xymatrix@R=2.5ex{
1 \ar[r] & \Gm \ar@{=}[d] \ar[r] & f^{\sigma}_*\Gm \ar[d]^(.42){\pi^*} \ar[r] & \widehat T \ar[d]^(.42){\Delta} \ar[r] & 1 \\
1 \ar[r] & \Gm \ar[r] & h_*\Gm \ar[r] & \smash[t]{\widehat T \oplus \widehat T}\vphantom{T} \ar@{=}[d] \ar[r] & 1 \\
1 \ar[r] & \Gm \oplus \Gm \ar[u]^(.46)\Pi \ar[r] & f_*\Gm \oplus g_*\Gm \ar[u] \ar[r] & \smash[t]{\widehat T \oplus \widehat T}\vphantom{T} \ar[r] & 1
}
\end{align*}
montre l'existence de $\beta \in \Br(X)$
tel que $\alpha^0 \boxplus \alpha = \pi^*\alpha^{\sigma} + h^*\beta$
dans $\Br_1(Y \times_k Z / X)$.
Quitte à remplacer $\alpha^0$ par $\alpha^0+f^*\beta$, on peut supposer
que $\alpha^0 \boxplus \alpha = \pi^*\alpha^{\sigma}$.
Comme $\alpha^\sigma \in \Br_\nr(Y^\sigma)$, il s'ensuit que
$\alpha^0\boxplus \alpha\in \Br_\nr(Y\times_k Z)$.
On en déduit que $\alpha^0 \in \Br_\nr(Y)$; en effet,
si~$\xi$ est un point de codimension~$1$
d'une compactification lisse de~$Y$,
le résidu de $\alpha^0\boxplus\alpha$ au point générique de $\xi\times_k Z$
est l'image du résidu de~$\alpha^0$ en~$\xi$
par
la flèche de restriction
$H^1(k(\xi),\Q/\Z)\to H^1(k(\xi \times_k Z),\Q/\Z)$
(voir \cite[Proposition~1.1.1]{ctsd94}), laquelle est injective
puisque~$Z$ est géométriquement irréductible.
Comme $\phi^\sigma(\alpha^\sigma)=\phi^0(\alpha^0)$,
on a maintenant montré que $\phi^\sigma(\alpha^\sigma) \in  \phi^0(\Br_\nr(Y))$,
ce qui conclut la preuve de~(ii).

Vérifions~(iii).
Soit $\beta \in \Br(X)$.
Soit~$K$ le corps des fonctions du torseur sous~$T$, sur~$k$,
déterminé par $\sigma \in Z^1(k,T)$.
Vu l'isomorphisme canonique $Y \otimes_k K = Y^\sigma \otimes_k K$
de schémas sur $X\otimes_kK$,
on a $(f^*\beta) \otimes_k K \in \Br_\nr(Y \otimes_kK)$
si et seulement si $((f^\sigma)^*\beta) \otimes_k K \in \Br_\nr(Y^\sigma\otimes_kK)$.
Par ailleurs,
un élément de $\Br(Y)$ appartient à $\Br_\nr(Y)$ si et seulement si
son image dans $\Br(Y \otimes_kK)$ appartient à $\Br_\nr(Y \otimes_kK)$,
puisque l'extension~$K/k$ est régulière;
et de même pour~$Y^\sigma$.
Donc $f^*\beta \in \Br_\nr(Y)$
si et seulement si $(f^\sigma)^*\beta \in \Br_\nr(Y^\sigma)$.
\end{proof}

\subsection{Preuve du théorème~\ref{th descente}}

Fixons $(x_v)_{v\in\Omega}\in X(\A_k)^A$ et exhibons $\sigma \in Z^1(k,T)$
tel que $(x_v)_{v\in\Omega} \in f^\sigma\Big(Y^\sigma(\A_k)^{\Br_\nr(Y^\sigma)}\Big)$.
Commençons par exploiter l'orthogonalité à $A \cap \Br_1(X)$
à l'aide de la théorie de la descente de Colliot-Thélène et Sansuc sous sa forme classique:

\begin{prop}
\label{prop:descente etape 1}
Il existe $\sigma \in Z^1(k,T)$ tel que
$(x_v)_{v\in\Omega} \in f^\sigma(Y^\sigma(\A_k))$.
\end{prop}

\begin{proof}
Il suffit d'appliquer
\cite[Theorem~4.1.1, Theorem~6.1.2]{skobook},
lorsque~$X$ est propre,
ou \cite[Theorem~8.4, Proposition~8.12]{haskoopendescent}, en général,
puisque les cup-produits d'un élément de $H^1(k,\widehat T)$ par $[f]\in H^1(X,T)$
appartiennent à $\Ker(f^*)$ donc à~$A$.
\end{proof}

Remplacer~$f$ par~$f^\sigma$ ne modifie pas le groupe~$A$,
d'après
la proposition~\ref{prop:descente prelim}~(iii),
et est donc loisible en vue d'établir le théorème.
La proposition~\ref{prop:descente etape 1} permet ainsi de supposer
que $(x_v)_{v\in\Omega} \in f(Y(\A_k))$.
On a alors
$(x_v)_{v\in\Omega} \in f^\sigma(Y^\sigma(\A_k))$
pour tout~$\sigma$ tel que $[\sigma] \in \Sha^1(k,T)$.
Pour tout tel~$\sigma$, considérons l'application
\begin{align*}
\epsilon_{\sigma}:\Sha^2(k,\widehat T) \cap \phi^0(\Br_\nr(Y)) \to \Q/\Z
\end{align*}
définie par $\epsilon_{\sigma}(\alpha')=\sum_{v \in \Omega} \inv_v \alpha^\sigma(y_v)$
pour
un $\alpha^\sigma \in \Br_\nr(Y^\sigma)$ tel que $\alpha'=\phi^\sigma(\alpha^\sigma)$
et un relèvement $(y_v)_{v \in \Omega} \in Y^{\sigma}(\A_k)$ de $(x_v)_{v\in\Omega}$.
L'existence de~$\alpha^\sigma$
est assurée par la proposition~\ref{prop:descente prelim}~(ii).
Que la somme $\sum_{v\in\Omega}\inv_v\alpha^\sigma(y_v)$ ne dépende pas du choix de~$(y_v)_{v\in\Omega}$
résulte de
la proposition~\ref{prop:descente prelim}~(iv),
compte tenu que $\alpha' \in \Sha^2(k,\widehat T)$.
Qu'elle ne dépende pas du choix de~$\alpha^\sigma$
vient de l'hypothèse que $(x_v)_{v \in\Omega} \in X(\A_k)^A$
et de la proposition~\ref{prop:descente prelim}~(iii).

\begin{prop}
\label{prop:variation epsilon sigma}
On a
$\epsilon_{\sigma}(\alpha')=\epsilon_0(\alpha') - \langle \alpha', [\sigma]\rangle$
pour tout $\alpha' \in \Sha^2(k,\widehat T) \cap \phi^0(\Br_\nr(Y))$
et tout $\sigma\in Z^1(k,T)$ tel que $[\sigma]\in\Sha^1(k,T)$,
où $\langle-,-\rangle: \Sha^2(k,\widehat T)\times \Sha^1(k,T) \to \Q/\Z$ désigne l'accouplement
de Poitou--Tate (voir \cite[Rappels~8.2]{sansuclinear} pour sa définition).
\end{prop}

\begin{proof}
Soit $\alpha^\sigma \in \Br_\nr(Y^\sigma)$ tel que $\alpha'=\phi^\sigma(\alpha^\sigma)$.
Comme on a vu dans la preuve
de la proposition~\ref{prop:descente prelim}~(ii),
dont on reprend les notations~$Z$, $\phi$, $\pi$,
il existe $\alpha^0 \in \Br_\nr(Y)$
et $\alpha \in \Br_1(Z)$
tels que $\alpha'=\phi^0(\alpha^0)=\phi(\alpha)$
et que
$\alpha^0 \boxplus \alpha = \pi^*\alpha^{\sigma}$
dans $\Br(Y \times_k Z)$.
Fixons $(z_v)_{v\in\Omega} \in Z(\A_k)$.
Pour $(y_v)_{v\in\Omega} \in Y(\A_k)$ relevant~$(x_v)_{v\in\Omega}$, on a
\begin{align*}
\epsilon_{\sigma}(\alpha')&=
\sum_{v\in\Omega} \inv_v (\pi^*\alpha^{\sigma})(y_v\times z_v)
=\sum_{v\in\Omega} \inv_v (\alpha^0 \boxplus \alpha)(y_v\times z_v)\\
&=\sum_{v\in\Omega} \inv_v \alpha^0 (y_v) +\sum_{v\in\Omega} \inv_v \alpha(z_v)
=\epsilon_0(\alpha') - \langle \alpha',[\sigma]\rangle\rlap{,}
\end{align*}
où la dernière égalité vient de \cite[Lemme~8.4]{sansuclinear}
(voir aussi \cite[Théorème~6.2.1]{skobook}).
\end{proof}

La flèche naturelle
$\Hom(\Sha^2(k,\widehat T),\Q/\Z) \to \Hom(\Sha^2(k,\widehat T) \cap \phi^0(\Br_\nr(Y)),\Q/\Z)$
étant surjective (puisque $\Q/\Z$ est un $\Z$\nobreakdash-module injectif)
et l'accouplement de Poitou--Tate étant parfait
(voir \cite[Theorem~8.6.7]{neukirchschmidtwingberg}),
il existe
$[\sigma] \in \Sha^1(k,T)$ tel que
$\epsilon_0(\alpha')= \langle \alpha', [\sigma]\rangle$
pour tout
$\alpha' \in \Sha^2(k,\widehat T)\cap\phi^0(\Br_\nr(Y))$.
D'après la proposition~\ref{prop:variation epsilon sigma},
on a alors $\epsilon_{\sigma}=0$.
Quitte à remplacer~$f$ par~$f^\sigma$, on peut donc supposer que $\epsilon_0=0$.
La proposition suivante conclut maintenant la preuve du théorème.

\begin{prop}
\label{prop:existence yv}
Si $\epsilon_0=0$,
alors $(x_v)_{v \in\Omega} \in f\big(Y(\A_k)^{\Br_\nr(Y)}\big)$.
\end{prop}

\begin{proof}
Fixons
$(y_v)_{v \in \Omega} \in Y(\A_k)$
relevant~$(x_v)_{v\in\Omega}$.
L'application
 $\Br_\nr(Y) \to \Q/\Z$,
 $\alpha^0 \mapsto \sum_{v \in \Omega} \inv_v \alpha^0(y_v)$
s'annule sur
le noyau de la flèche
$\Br_\nr(Y) \to H^2(k,\widehat T)/\Sha^2(k,\widehat T)$
induite par~$\phi^0$,
puisque $\epsilon_0=0$.
Comme~$\Q/\Z$ est un $\Z$\nobreakdash-module injectif, elle se factorise donc par
$H^2(k,\widehat T)/\Sha^2(k,\widehat T)$.
Il s'ensuit, compte tenu de la suite exacte
\begin{align*}
\xymatrix{
T(\A_k) \ar[r]& \Hom\big(H^2(k,\widehat T),\Q/\Z\big) \ar[r]& \Hom\big(\Sha^2(k,\widehat T),\Q/\Z\big)
}
\end{align*}
établie dans \cite[Théorème~2]{harariapproxforte}, qu'il existe $(t_v)_{v\in\Omega}\in T(\A_k)$
tel que $\sum_{v\in\Omega} \inv_v \alpha^0(y_v)=\sum_{v\in\Omega} \inv_v (\phi^0(\alpha^0)\smile t_v)$
pour tout $\alpha^0\in \Br_\nr(Y)$.
On a alors $(t_v^{-1}\cdot y_v)_{v\in\Omega} \in Y(\A_k)^{\Br_\nr(Y)}$
d'après la proposition~\ref{prop:descente prelim}~(iv);
de plus, ce point adélique relève encore~$(x_v)_{v\in\Omega}$.
\end{proof}

\section{D'un torseur universel à l'autre}
\label{sec:duntorseuralautre}

La théorie de la descente, originellement développée pour des variétés propres,
a depuis été utilement appliquée, dans le cadre de l'étude des points rationnels,
à des variétés ouvertes
(voir \cite{ctskodescent}, \cite{hbsko}, \cite{ctbudapest}, \cite{cthasko}, \cite{dsw}).
Étant donnés une variété~$X$ propre, lisse et rationnellement connexe
et un ouvert $V \subset X$ tel que $\bark[V]^*=\bark{}^*$, la question se pose naturellement
de comparer la descente appliquée à~$X$ avec la descente appliquée à~$V$,
donc les torseurs universels de~$X$ avec ceux de~$V$.
C'est à cette question que ce paragraphe est consacré, dans le cadre légèrement plus général des torseurs
de type quelconque sous un groupe de type multiplicatif, sans supposer~$X$ propre.

Fixons, jusqu'à la fin du~\textsection\ref{sec:duntorseuralautre}, un corps~$k$
de caractéristique nulle, une clôture algébrique~$\bark$ de~$k$,
une variété~$X$ lisse et géométriquement irréductible sur~$k$,
un ouvert dense $V \subset X$
tel que $\bark[V]^*=\bark{}^*$,
un groupe abélien de type fini~$\Rhat$ muni d'une action
continue de $\Gal(\bark/k)$
et un homomorphisme $\Gal(\bark/k)$\nobreakdash-équivariant $\lambda:\Rhat\to \Pic(V_{\bark})$.

Définissons des $\Gal(\bark/k)$\nobreakdash-modules discrets $\widehat T$
et $\widehat Q$ et un homomorphisme $\nu:\widehat T \to \Pic(X_{\bark})$
par le diagramme commutatif à lignes exactes
\begin{align}
\begin{aligned}
\label{diag:qtr}
\xymatrix@R=2ex{
0 \ar[r] & \widehat Q \ar@{=}[d] \ar[r] & \widehat T \ar[d]^(.4){\nu} \ar[r] & \Rhat \ar[d]^(.4)\lambda \ar[r] & 0 \\
0 \ar[r] & \Div_{X_{\bark}\setminus V_{\bark}}(X_{\bark}) \ar[r] & \Pic(X_{\bark}) \ar[r] & \Pic(V_{\bark}) \ar[r] & 0\rlap{\text{,}}
}
\end{aligned}
\end{align}
où $\Div_{X_{\bark}\setminus V_{\bark}}(X_\bark)$ désigne le groupe des diviseurs sur~$X_{\bark}$
supportés par $X_{\bark}\setminus V_{\bark}$;
ainsi~$\widehat T$ est par définition le produit fibré de~$\Rhat$
par~$\Pic(X_{\bark})$ au-dessus de~$\Pic(V_\bark)$.
Comme
le groupe $\widehat Q=\Div_{X_{\bark}\setminus V_{\bark}}(X_\bark)$ admet une base sur~$\Z$ permutée par~$\Gal(\bark/k)$,
il donne naissance à
 un tore quasi-trivial
$Q=\Homrond(\widehat Q,\Gm)$ sur~$k$.
Posant
 $R=\Homrond(\Rhat,\Gm)$
et
$T=\Homrond(\widehat T,\Gm)$,
on a une suite exacte
de groupes de type multiplicatif sur~$k$:
\begin{align}
\label{eq:rtq}
\xymatrix{
1 \ar[r]&R \ar[r]& T \ar[r]& Q \ar[r]& 1\rlap{.}
}
\end{align}

Rappelons que le \emph{type} d'un torseur $g:Z \to V$ sous~$R$
est la classe d'isomorphisme du torseur $g_{\bark}:Z_{\bark}\to V_{\bark}$ sous~$R_{\bark}$.
C'est un élément $\Gal(\bark/k)$\nobreakdash-invariant de $H^1(V_{\bark},R_{\bark})$.
De façon équivalente
et conformément à la définition de~\cite[\textsection2.0]{ctsandescent2},
l'isomorphisme canonique
\begin{align}
\label{eq:identification type}
H^1(V_{\bark},R_{\bark})=H^1(V_\bark,\Homrond(\Rhat,\Gm))=\Hom(\Rhat,H^1(V_\bark,\Gmbark))=\Hom(\Rhat,\Pic(V_{\bark}))
\end{align}
(où l'isomorphisme du milieu résulte, dans le cas où $\Rhat=\Z/n\Z$,
de l'hypothèse $\bark[V]^*=\bark{}^*$)
applique le type de~$g$ sur l'homomorphisme $\Rhat\to \Pic(V_\bark)$
qui à $\chi \in \Rhat = \Hom(R_{\bark},\Gmbark)$
associe
l'image de la classe de~$g_\bark$ par $\chi_*:H^1(V_\bark,R_\bark)\to H^1(V_\bark,\Gmbark)=\Pic(V_\bark)$.

La partie~(ii) de la proposition suivante constitue l'énoncé principal du~\textsection\ref{sec:duntorseuralautre}.

\begin{prop}
\label{prop:comparaisontorseurs}
Soit $f:Y \to X$ un torseur sous~$T$, de type~$\nu$.
Soit $W=f^{-1}(V)$.
\begin{enumerate}
\item[(i)] Le torseur $W/R \to V$ sous~$Q$ est trivial.  Autrement dit, il existe 
 un $V$\nobreakdash-isomorphisme $Q$\nobreakdash-équivariant $W/R \simeq V \times_k Q$.
\end{enumerate}
Fixons un tel isomorphisme et
notons $\pi: W \to Q$
le morphisme $T$\nobreakdash-équivariant obtenu en le composant avec le morphisme quotient $W \to W/R$
et la seconde projection $V \times_k Q \to Q$.
\begin{enumerate}
\item[(ii)]
Pour  $q \in Q(k)$,
l'action naturelle de~$R$ sur $W_q=\pi^{-1}(q)$
fait du morphisme $W_q \to V$ induit par~$f$ un torseur, sous~$R$, de type~$\lambda$.
\item[(iii)]
Les classes d'isomorphisme de torseurs $V' \to V$ sous~$R$, de type~$\lambda$, tels que le produit contracté
$V' \times^R_k T \to V$ soit isomorphe, en tant que torseur sous~$T$, à $W \to V$,
sont exactement les classes d'isomorphisme des torseurs $W_q \to V$ pour $q \in Q(k)$.
\end{enumerate}
\end{prop}

\begin{proof}
Considérons le diagramme commutatif
\begin{align}
\begin{aligned}
\label{diag:comparaison torseurs}
\xymatrix@R=2ex{
& H^1(X,T) \ar[dr] \ar[d] \\
& H^1(X_\bark,T_\bark) \ar[dr]^\rho & H^1(V,T) \ar[r] \ar[d] & H^1(V,Q)
\ar@{}|{\rotatebox{-90}{$\hookrightarrow$}}[d] \\
0 \ar[r] & H^1(V_\bark,R_\bark) \ar[r]^\iota & H^1(V_\bark,T_\bark) \ar[r] & H^1(V_\bark,Q_\bark)\rlap{,}
}
\end{aligned}
\end{align}
dont la ligne du bas est exacte
puisque $\bark[V]^*=\bark{}^*$.
La flèche verticale de droite est injective car son noyau est $H^1(k,Q)$, qui est nul
d'après le théorème de Hilbert~90 puisque~$Q$ est un tore quasi-trivial.
Voyant~$\lambda$ et~$\nu$ comme des éléments de $H^1(V_\bark,R_\bark)$
et de $H^1(X_\bark,T_\bark)$, on constate, en comparant la définition de~$\nu$
(voir le diagramme~\eqref{diag:qtr})
avec les trois isomorphismes canoniques
$H^1(X_{\bark},T_{\bark})=\Hom(\widehat T,\Pic(X_{\bark}))$,
$H^1(V_{\bark},T_{\bark})=\Hom(\widehat T,\Pic(V_{\bark}))$,
$H^1(V_{\bark},R_{\bark})=\Hom(\Rhat,\Pic(V_{\bark}))$
(voir~\eqref{eq:identification type}),
que
\begin{align}
\label{eq:rho nu iota lambda}
\rho(\nu)=\iota(\lambda)\rlap{.}
\end{align}
Comme $[f] \in H^1(X,T)$ s'envoie sur $\nu \in H^1(X_\bark,T_\bark)$,
on déduit
de~\eqref{diag:comparaison torseurs}
et~\eqref{eq:rho nu iota lambda} que l'image de~$[f]$ dans $H^1(V,Q)$ s'annule.
Cette annulation équivaut à l'assertion~(i).

Fixons un $V$\nobreakdash-isomorphisme $Q$\nobreakdash-équivariant $W/R \simeq V \times_k Q$.
Le morphisme $W_q \to V$ induit par~$f$
pour $q \in Q(k)$
est un torseur sous~$R$ puisque c'est la restriction, au-dessus de $V\times \{q\}$, du
torseur $W \to W/R = V\times_k Q$ sous~$R$.
Notant $W_q \times_k^R T$ le produit contracté de~$W_q$ et~$T$ sous l'action de~$R$,
l'inclusion $R$\nobreakdash-équivariante de~$W_q$ dans~$W$ induit un morphisme
$W_q \times_k^R T \to W$ de torseurs sous~$T$.
Comme tout morphisme de torseurs sous~$T$,
c'est un isomorphisme.
L'application~$\iota$ envoie donc le type de~$W_q$ sur le type de~$W$,
c'est-à-dire sur~$\rho(\nu)$.  Ainsi, vu~\eqref{eq:rho nu iota lambda}
et vu l'injectivité de~$\iota$
(voir~\eqref{diag:comparaison torseurs}), le type de~$W_q$ est~$\lambda$;
d'où~(ii).

Pour vérifier~(iii), remarquons que $W_1\to V$ (où $1 \in Q(k)$ désigne le neutre)
est l'un des torseurs remplissant les conditions de~(iii).
Au vu de la suite exacte
\begin{align}
\xymatrix{
Q(k) \ar[r] & H^1(V,R) \ar[r] & H^1(V,T)\rlap{,}
}
\end{align}
qui résulte de~\eqref{eq:rtq} et de l'hypothèse que $\bark[V]^*=\bark{}^*$,
il s'ensuit que tout $V' \to V$ remplissant les conditions de~(iii)
est isomorphe à $W_1 \times_k^R T_q \to V$ pour un $q \in Q(k)$,
où $T_q$ désigne la fibre de $T\to Q$ en~$q$.
Or l'action de~$T$ sur~$W$ induit un isomorphisme $W_1 \times_k^R T_q \to W_q$.
\end{proof}

\begin{cor}
\label{cor:torseurs universels}
Supposons $\Pic(X_\bark)$ de type fini.
Tout torseur universel de~$X$ contient un ouvert dense admettant un morphisme
lisse, vers un tore quasi-trivial, dont les fibres sont des torseurs universels
de~$V$.
\end{cor}

\begin{proof}
Appliquer la proposition~\ref{prop:comparaisontorseurs}~(ii)
en prenant pour~$\lambda$ et~$\nu$ l'identité.
\end{proof}

Dans certains cas intéressants, le morphisme $\pi:W\to Q$ de la proposition~\ref{prop:comparaisontorseurs}
admet sur~$\bark$ une section ou au moins une section rationnelle.
La proposition suivante interviendra dans la preuve du théorème~\ref{th points rat hr}.

\begin{prop}
\label{prop:sectionbark}
Reprenons les notations de la proposition~\ref{prop:comparaisontorseurs}
et supposons $k=\bark$.
\begin{enumerate}
\item[(i)] Si $\Rhat$ est sans torsion, alors~$\pi$ admet une section.
\item[(ii)]
Notons $V' \to V$ un torseur de type~$\lambda$.
Si le sous-groupe de torsion de~$\Rhat$ est cyclique et si les compactifications lisses
de~$V'$
sont rationnellement connexes,
alors~$\pi$ admet une section rationnelle.
\end{enumerate}
\end{prop}

Dans~(ii), le torseur $V' \to V$ est unique à isomorphisme près puisque $k=\bark$.

\begin{proof}
Si~$\Rhat$ est sans torsion,
la première ligne du diagramme~\eqref{diag:qtr} est scindée,
donc $T \simeq R \times_k Q$,
donc $W = W_1 \times_k^R T \simeq W_1 \times_k Q$,
compte tenu de la proposition~\ref{prop:comparaisontorseurs}~(iii).
(Comme dans la preuve de la proposition~\ref{prop:comparaisontorseurs},
nous notons ici $W_1=\pi^{-1}(1)$, où $1 \in Q(k)$ désigne le neutre.)
Le morphisme~$\pi$ s'identifie alors à la seconde projection $W_1\times_k Q \to Q$.
Le choix d'un $k$\nobreakdash-point de~$W_1$ détermine une section de~$\pi$.
D'où~(i).

Prouvons~(ii).
Supposons le sous-groupe de torsion de~$\Rhat$ cyclique
et choisissons un isomorphisme $\Rhat \simeq \Z^n \oplus \Z/m\Z$ pour un $n\geq 0$ et un $m\geq 1$
ainsi que des relèvements dans~$\widehat T$
des~$n$ générateurs canoniques de~$\Z^n \subset \Rhat$
et du générateur $0\oplus 1$ du sous-groupe
de torsion de~$\Rhat$.
Ces choix déterminent un diagramme commutatif à lignes exactes
\begin{align}
\begin{aligned}
\xymatrix@R=2ex{
0 \ar[r] & \Z \ar[d] \ar[r] & \Z^n \oplus \Z \ar[d] \ar[r] & \Z^n \oplus \Z/m\Z \ar[r] \ar@<-.065em>@{=}[d] & 0 \\
0 \ar[r] & \widehat Q \ar[r] & \widehat T \ar[r] & \Rhat \ar[r] & 0\rlap{,}
}
\end{aligned}
\end{align}
d'où, dualement, un diagramme commutatif à lignes exactes
\begin{align}
\label{diag:rtq cyclique}
\begin{aligned}
\xymatrix@R=2ex{
1 \ar[r] &  R \ar@{=}[d] \ar[r] & T \ar[d] \ar[r] &  Q \ar[r] \ar[d] & 1 \\
1 \ar[r] & \Gm^n \times_k \mmu_m \ar[r] & \Gm^n \times_k \Gm \ar[r] & \Gm \ar[r] & 1\rlap{.}
}
\end{aligned}
\end{align}
Appliquant le foncteur $W\times^T_k-$
au carré de droite de~\eqref{diag:rtq cyclique}, on obtient le carré de gauche du diagramme commutatif
\vspace*{-2ex}
\begin{align}
\label{diag:pi pi0}
\begin{aligned}
\xymatrix@R=2ex{
\ar@<.3em>@{..>}@/^1.2pc/[rr]^\pi
W \ar[d] \ar[r] & V \times_k Q \ar[d] \ar[r] & Q \ar[d] \\
W \times_k^T (\Gm^n \times_k \Gm) \ar[r] & V \times_k \Gm \ar[r] & \Gm\rlap{,}
}
\end{aligned}
\end{align}
dont les deux carrés sont cartésiens, dont la flèche verticale de droite est la flèche
verticale de droite de~\eqref{diag:rtq cyclique} et dont les flèches horizontales de droite sont les projections.
Supposons les compactifications lisses de~$V'$ rationnellement connexes.
D'après la proposition~\ref{prop:comparaisontorseurs}~(ii),
la fibre générique géométrique de~$\pi$ se déduit de~$V'$ par extension des scalaires.
Par ailleurs, elle se déduit par extension des scalaires
de la fibre générique géométrique
du morphisme
$\pi_0:W \times_k^T (\Gm^n\times_k \Gm) \to \Gm$
issu de~\eqref{diag:pi pi0}.
Les compactifications lisses de cette dernière sont donc rationnellement connexes.
Comme~$\Gm$ est une courbe sur un corps algébriquement clos de caractéristique nulle,
il s'ensuit, grâce au théorème de Graber--Harris--Starr \cite[Theorem~1.1]{ghs}
combiné à \cite[Chapter~IV, Theorem~6.10]{kollarbook},
que~$\pi_0$ admet une section rationnelle.
Par conséquent~$\pi$ admet une section rationnelle.
\end{proof}

\begin{rmk}
Lorsque~$V'$ est un espace homogène d'un groupe algébrique linéaire,
comme ce sera le cas dans toutes les applications de la proposition~\ref{prop:sectionbark}
contenues dans cet article,
le cas particulier du théorème de Graber--Harris--Starr utilisé ci-dessus
est connu depuis Springer (voir \cite[Chapitre~III,
\textsection2.3, Théorème~1$'$ et \textsection2.4, Corollaire~1]{serrecg}).
\end{rmk}

\section{Fibrations au-dessus de tores quasi-triviaux}
\label{sec:fibrations tore qtriv}

Nous rassemblons dans ce paragraphe des énoncés (certains connus, d'autres nouveaux)
concernant les points rationnels ou les zéro-cycles
applicables à
l'espace total de fibrations propres et lisses en variétés rationnellement connexes
au-dessus d'un tore quasi-trivial.
Il s'agit là de fibrations au-dessus de~$\P^n_k$ dont
le lieu des fibres singulières est géométriquement
une réunion de $n+1$ hyperplans.  Comme on le sait depuis \cite{skorodescent},
on peut relâcher l'hypothèse de lissité et la remplacer
par la condition que les fibres au-dessus des points
de codimension~$1$ du tore quasi-trivial sont scindées.
Rappelons qu'une variété est dite \emph{scindée} si l'une au moins
de ses composantes irréductibles
est géométriquement irréductible et de multiplicité~$1$
(notion introduite dans \emph{op.\ cit.}, Definition~0.1).

Fixons, dans tout le~\textsection\ref{sec:fibrations tore qtriv},
un corps de nombres~$k$,
une
$k$\nobreakdash-algèbre étale~$E$ et une base de~$E$ comme espace vectoriel sur~$k$.
Posons $Q=R_{E/k}\Gm$, $Q^\aff=R_{E/k}\A^1_E$ et $n=[E:k]$.
Le choix de la base fournit un isomorphisme $Q^\aff=\A^n_k$, d'où une suite d'inclusions
$Q \subset Q^\aff \subset \P^n_k$.
Fixons enfin une variété~$X$ irréductible, propre et lisse sur~$k$
et un morphisme dominant $f:X \to \P^n_k$ dont la fibre générique est rationnellement connexe.

\subsection{Fibres presque toutes scindées}

La situation la plus favorable est celle où les fibres de~$f$ au-dessus des points
de codimension~$1$ de~$Q^\aff$ sont toutes scindées.
Le théorème suivant est établi dans \cite[Theorem~1]{skofibration} et \cite[Theorem~2.1]{skorodescent}, pour~(i),
et dans \cite[Théorème~3.2.1 et la remarque à la fin du \textsection3.2]{hararifleches},
pour~(ii).
Les hypothèses de \emph{loc.\ cit.}\ sont ici satisfaites en vertu de \cite[Theorem~1.1]{ghs}.
Pour $q\in Q$, notons $X_q=f^{-1}(q)$.

\begin{thm}[Skorobogatov (i), Harari (ii)]
\label{th:fibration skohar}
Supposons les fibres de~$f$ au-dessus des points de codimension~$1$ de~$Q^\aff$ scindées.
\begin{enumerate}
\item[(i)]
Si~$X_q(k)$ est dense dans $X_q(\A_k)$ pour tout point rationnel~$q$ d'un sous-ensemble
hilbertien de~$Q$, alors~$X(k)$ est dense dans $X(\A_k)$.
\item[(ii)]
Si $X_q(k)$ est dense dans $X_q(\A_k)^{\Br(X_q)}$
pour tout point rationnel~$q$ d'un sous-ensemble hilbertien de~$Q$,
alors~$X(k)$ est dense dans $X(\A_k)^{\Br(X)}$.
\end{enumerate}
\end{thm}

Le théorème~\ref{th:fibration skohar} servira dans la démonstration du
théorème~\ref{thm:fibration presque favorable} ci-dessous.

\subsection{Fibres presque toutes scindées sur \texorpdfstring{$\bark$}{k̅}}

Le morphisme $f\otimes_k\bark:X\otimes_k\bark\to\P^n_\bark$
admet une section lorsque $n=1$, d'après \cite[Theorem~1.1]{ghs}.
Dans ce cas, ses fibres sont donc scindées.
Lorsque $n>1$, en revanche,
la condition que les fibres de $f\otimes_k\bark$ au-dessus des points de codimension~$1$
de $Q^\aff \otimes_k \bark$ sont scindées
est une condition forte, rarement satisfaite en pratique.
Sous cette hypothèse, on peut établir l'énoncé inconditionnel suivant,
qui servira dans la démonstration du théorème~\ref{th:66} ci-dessous (c'est-à-dire
du théorème~\ref{th points rat hr}).

\begin{thm}
\label{thm:fibration presque favorable}
Supposons les fibres de~$f$ au-dessus des points de codimension~$1$ de~$Q$
et les fibres de $f\otimes_k\bark$ au-dessus des points de codimension~$1$ de $Q^\aff\otimes_k\bark$
scindées.
\begin{enumerate}
\item[(i)]
Si~$X_q(k)$ est dense dans $X_q(\A_k)$ pour tout point rationnel~$q$ d'un sous-ensemble
hilbertien de~$Q$, alors~$X(k)$ est dense dans $X(\A_k)^{\Br_1(X)}$.
\item[(ii)]
Si $X_q(k)$ est dense dans $X_q(\A_k)^{\Br(X_q)}$
pour tout point rationnel~$q$ d'un sous-ensemble hilbertien de~$Q$,
alors~$X(k)$ est dense dans $X(\A_k)^{\Br(X)}$.
\end{enumerate}
\end{thm}

\begin{proof}
Comme~$Q^\aff$ ne contient qu'un nombre fini de points de codimension~$1$ hors de~$Q$,
il existe une extension finie $\ell/k$ telle que les fibres de $f\otimes_k\ell$
au-dessus des points de codimension~$1$ de $Q^\aff \otimes_k\ell$ soient scindées.
Soit $Q'=R_{E\otimes_k\ell/k}\Gm \subset Q'^\aff = R_{E\otimes_k\ell/k}\A^1$.
Pour $r\in Q(k)$,
soit $b^{\aff(r)}:Q'^\aff \to Q^\aff$ le morphisme défini par $b^{\aff(r)}(x)=rN_{E\otimes_k\ell/E}(x)$.
Le choix d'une base de~$\ell$ sur~$k$ définit une compactification $Q'^\aff \subset \P^{n'}_k$,
avec $n'=\dim(Q')$.
Définissons $V^\aff$, $V'^{\aff (r)}$ et $c^{\aff(r)}$
par le diagramme suivant, dont tous les carrés sont cartésiens
et dans lequel $X'^{(r)}$
est une compactification lisse arbitraire
de $V'^{\aff(r)}$ telle que~$f'^{(r)}$
soit un morphisme:
\begin{align*}
\xymatrix@C=.5em{
X'^{(r)} \ar[d]_(.45){f'^{(r)}} & \ar@{}[l]|*{\supset} V'^{\aff(r)} \ar[rrr]^(.55){c^{\aff(r)}} \ar[d] &&& V^\aff \ar[d] \ar@{}[r]|*{\subset} & X \ar[d]^(.45)f \\
\P^{n'}_k                            & \ar@{}[l]|*{\supset} Q'^\aff  \ar[rrr]^(.55){b^{\aff(r)}}        &&& Q^\aff   \ar@{}[r]|*{\subset} & \P^n_k \rlap{.}
}
\end{align*}

\begin{lem}
\label{lem:fibres scindees}
Pour tout $r\in Q(k)$,
les fibres de~$f'^{(r)}$ au-dessus des points de codimension~$1$
de~$Q'^\aff$ sont scindées.
\end{lem}

\begin{proof}
Soit $\xi \in Q'^\aff$ un point de codimension~$1$.
Le point $b^{\aff(r)}(\xi) \in Q^\aff$ est de codimension~$\leq 1$
puisque $(b^{\aff(r)})^{-1}(b^{\aff(r)}(\xi))$ contient~$\xi$ et est,
comme chaque fibre de~$b^{\aff(r)}$, de dimension $\dim(Q'^\aff)-\dim(Q^\aff)$.
Si $\xi \in Q'$, on a de plus $b^{\aff(r)}(\xi)\in Q$, de sorte que $f^{-1}(b^{\aff(r)}(\xi))$ est scindée,
donc $(f'^{(r)})^{-1}(\xi)$ aussi.
Si $\xi\notin Q'$, alors~$\ell$ se plonge $k$\nobreakdash-linéairement
dans le corps résiduel de~$\xi$.
Pour vérifier, dans ce cas, que
$(f'^{(r)})^{-1}(\xi)$ est scindée, il est donc loisible d'étendre les scalaires de~$k$ à~$\ell$,
ce qui permet de supposer que $\ell=k$.
La variété $f^{-1}(b^{\aff(r)}(\xi))$ est alors elle-même scindée, par définition de~$\ell$.
\end{proof}

Posons $V=f^{-1}(Q)$ et $V'^{(r)}=(f'^{(r)})^{-1}(Q')$ et notons $b^{(r)}:Q'\to Q$ et $c^{(r)}:V'^{(r)}\to V$
les morphismes induits par~$b^{\aff(r)}$ et~$c^{\aff(r)}$.  Posons $T=\Ker(b^{(1)})=\Ker\mkern1mu(N_{E\otimes_k\ell/E}:Q'\to Q)$.

\begin{lem}
\label{lem:cr torseur}
Pour $r\in Q(k)$, le morphisme $c^{(r)}$ est un torseur sous le tore~$T$.
Tout tordu de~$c^{(1)}$ (où $1 \in Q(k)$ est le neutre)
par un cocycle de $Z^1(k,T)$ est isomorphe à l'un des~$c^{(r)}$.
\end{lem}

\begin{proof}
En effet $b^{(r)}$ est un torseur sous~$T$ et
les classes de~$b^{(r)}$ et de~$b^{(1)}$ dans $H^1(Q,T)$ diffèrent par l'image
de~$r$ par l'application bord $Q(k)\to H^1(k,T)$, laquelle est surjective puisque $H^1(k,Q')=0$ en vertu
du théorème de Hilbert~90.
\end{proof}

Nous sommes maintenant en position d'établir~(i) et~(ii).

Supposons~$X_q(k)$ dense dans~$X_q(\A_k)$ pour tout point rationnel~$q$ d'un sous-ensemble hilbertien de~$Q$.
Fixons un point adélique $(P_v)_{v\in\Omega} \in X(\A_k)^{\Br_1(X)}$ et montrons qu'on peut l'approcher
par un point rationnel de~$X$.
Soit $A \subset \Br_1(V)$ le sous-groupe formé des
cup-produits d'un élément de $H^1(k,\widehat T)$ par $[c^{(1)}]\in H^1(V,T)$.
Comme~$A$ est fini,
le ``lemme formel'' de Harari
assure qu'il existe $(P'_v)_{v\in\Omega} \in V(\A_k)^A$
arbitrairement proche de~$(P_v)_{v\in\Omega}$
dans~$X(\A_k)$
(voir \cite[Proposition~1.1]{ctskodescent}).
D'après \cite[Theorem~8.4, Proposition~8.12]{haskoopendescent} appliqué à $c^{(1)}:V'^{(1)}\to V$
et d'après le lemme~\ref{lem:cr torseur},
il existe $r\in Q(k)$ et $(P''_v)_{v\in\Omega} \in V'^{(r)}(\A_k)$
relevant~$(P'_v)_{v\in\Omega}$.
Grâce au lemme~\ref{lem:fibres scindees}
et compte tenu que l'image réciproque par~$b^{(r)}$ d'un sous-ensemble hilbertien de~$Q$
est un sous-ensemble hilbertien de~$Q'$,
le théorème~\ref{th:fibration skohar}~(i) est applicable à~$f'^{(r)}$.
Il existe donc $P'' \in V'^{(r)}(k)$ arbitrairement proche,
dans $X'^{(r)}(\A_k)$, de~$(P''_v)_{v\in\Omega}$.
Alors $c^{(r)}(P'') \in X(k)$
est arbitrairement proche de~$(P_v)_{v\in\Omega}$ dans~$X(\A_k)$.

Supposons maintenant~$X_q(k)$ dense dans $X_q(\A_k)^{\Br(X_q)}$ pour tout point rationnel~$q$ d'un sous-ensemble hilbertien de~$Q$,
fixons un point adélique $(P_v)_{v\in\Omega} \in X(\A_k)^{\Br(X)}$ et montrons qu'on peut l'approcher
par un point rationnel de~$X$.
D'après le corollaire~\ref{cor:descente}
appliqué à $c^{(1)}:V'^{(1)}\to V$
et d'après le lemme~\ref{lem:cr torseur},
il existe $r\in Q(k)$ et $(P''_v)_{v\in\Omega} \in V'^{(r)}(\A_k)^{\Br_\nr(V'^{(r)})}$
tels que $(c^{(r)}(P''_v))_{v\in\Omega}$ soit arbitrairement proche
de~$(P_v)_{v\in\Omega}$ dans~$X(\A_k)$.
Comme au paragraphe précédent, le lemme~\ref{lem:fibres scindees}
permet d'appliquer le théorème~\ref{th:fibration skohar}~(ii) à~$f'^{(r)}$.
Il en résulte l'existence de $P'' \in V'^{(r)}(k)$
arbitrairement proche,
dans $X'^{(r)}(\A_k)$, de~$(P''_v)_{v\in\Omega}$.
Le point $c^{(r)}(P'') \in X(k)$
est alors arbitrairement proche de~$(P_v)_{v\in\Omega}$ dans~$X(\A_k)$.
\end{proof}

\begin{rmk}
Si~$X$ est une variété propre et lisse sur~$k$,
si $V\subset X$ est un ouvert dense tel que $\bark[V]^*=\bark{}^*$
et si les groupes $\Pic(X_{\bark})$ et $\Pic(V_{\bark})$ sont sans torsion,
le théorème~\ref{thm:fibration presque favorable}~(i)
et la proposition~\ref{prop:sectionbark}~(i)
permettent de justifier l'implication
suivante, énoncée dans \cite[Remark~3.9]{wslc}:
si tout torseur universel de~$V$ vérifie l'approximation faible,
alors tout torseur universel de~$X$ vérifie l'approximation faible.
En effet, fixons un torseur universel~$Y$ de~$X$.
La proposition~\ref{prop:comparaisontorseurs}
fournit un ouvert dense $W\subset Y$, un tore quasi-trivial~$Q$ sur~$k$
et un morphisme lisse $\pi:W \to Q$ dont les fibres sont des torseurs universels de~$V$.
La proposition~\ref{prop:sectionbark}~(i)
assure que $\pi\otimes_k\bark$ admet une section
et donc que si $\pi':Z \to \P^n_k$ désigne une fibration compactifiant~$\pi$,
les fibres de $\pi'\otimes_k\bark$ au-dessus des points de codimension~$1$ de $Q^\aff\otimes_k\bark$
sont scindées.  Par le théorème~\ref{thm:fibration presque favorable}~(i),
il s'ensuit que si tout torseur universel de~$V$ vérifie l'approximation faible,
alors $Z(k)$ est dense dans $Z(\A_k)^{\Br_1(Z)}$.
Or $\Br_1(Z)=\Br_0(Z)$ (voir \cite[Th\'eor\`eme~2.1.2]{ctsandescent2}): donc~$Z$
et~$Y$
vérifient l'approximation faible.
\end{rmk}

\subsection{Fibres non scindées sur \texorpdfstring{$\bark$}{k̅}}

Les deux énoncés suivants ne font aucune
hypothèse sur les fibres de~$f$ au-dessus des points de codimension~$1$
de~$Q^\aff$.
Le théorème~\ref{thm:fibration cas general pointsrat} servira dans la preuve
du théorème~\ref{th:pointsrationnelsconjectural} ci-dessous;
le théorème~\ref{thm:fibration cas general zerocycles}, dans celle
du théorème~\ref{th but}.

\newcommand{\citehwnvc}{\cite[Corollary~9.25]{hw}}
\begin{thm}[\citehwnvc]
\label{thm:fibration cas general pointsrat}
Si la conjecture \cite[Conjecture~9.1]{hw} est vraie pour~$k$ et si $X_q(k)$ est dense dans $X_q(\A_k)^{\Br(X_q)}$
pour tout point rationnel~$q$ d'un sous-ensemble hilbertien de~$Q$,
alors~$X(k)$ est dense dans $X(\A_k)^{\Br(X)}$.
\end{thm}

\newcommand{\citehwhq}{\cite[Corollary~8.4~(1)]{hw}}
\begin{thm}[\citehwhq]
\label{thm:fibration cas general zerocycles}
Si~$X_q$ vérifie la conjecture~$(E)$
pour tout point fermé~$q$ d'un sous-ensemble hilbertien de~$Q$,
alors~$X$ vérifie la conjecture~$(E)$.
\end{thm}

Il paraît raisonnable d'espérer que le théorème~\ref{thm:fibration cas general pointsrat}
puisse être rendu inconditionnel dans le cas où les fibres de~$f$ au-dessus des points
de codimension~$1$ de~$Q$ sont scindées (hypothèse satisfaite dans toutes les applications
de ce théorème envisagées dans le présent article).

\section{Revêtements étales des espaces homogènes}
\label{sec:pi1eh}

Dans tout le~\textsection\ref{sec:pi1eh}, on fixe un corps~$k$ de caractéristique nulle,
une clôture algébrique~$\bark$ de~$k$ et un groupe algébrique~$G$ connexe, semi-simple et
simplement connexe sur~$k$.

\subsection{Action extérieure galoisienne sur le stabilisateur}
\label{subsec:action exterieure}

Soit~$V$ un espace homogène de~$G$.
Fixons $\bar v \in V(\bark)$ et supposons son stabilisateur $H_{\bar v} \subset G(\bark)$
fini.
Comme $\pi_1^\et(G_\bark,1)=1$ et comme $V_\bark=G_\bark/H_{\bar v}$,
on a canoniquement $\pi_1^\et(V_\bark,\bar v)=H_{\bar v}$.
La suite exacte fondamentale
\begin{align}
\label{se:pi1}
\xymatrix{
1 \ar[r] & \pi_1^{\et}(V_\bark,\bar v) \ar[r] & \pi_1^{\et}(V,\bar v)\ar[r] & \Gal(\bark/k) \ar[r] & 1
}
\end{align}
induit
donc une action extérieure canonique de $\Gal(\bark/k)$ sur~$H_{\bar v}$.
Celle-ci induit une action continue de $\Gal(\bark/k)$ sur $H_{\bar v}^\ab$.
Remarquons
qu'une action extérieure naturelle de~$\Gal(\bark/k)$
sur~$H_{\bar v}$ existerait même si nous n'avions pas supposé~$G$ simplement connexe
(voir \cite[\textsection2.3]{demarchelucchinireduction}).
Remarquons, d'autre part,
que la classe d'isomorphisme de~$H_{\bar v}$ vu comme groupe fini
muni d'une action extérieure de~$\Gal(\bark/k)$
ne dépend pas du choix de~$\bar v$.

\subsection{Groupe de Picard}
\label{subsec:stab picard}

Comme~$G$ est semi-simple, le lemme de Rosenlicht assure que $\bark[G]^*=\bark{}^*$ et
donc $\bark[V]^*=\bark{}^*$.
Il s'ensuit que $H^1(V_{\bark},\Q/\Z(1))=\Pic(V_{\bark})_{\text{tors}}$,
par théorie de Kummer.
D'autre part,
le groupe $\Pic(V_{\bark})$ est de torsion
puisque~$G_\bark$ est un revêtement de~$V_{\bark}$
et que $\Pic(G_{\bark})=0$ (voir
\cite[\textsection4.3, Theorem~1]{voskbirinv}).
De ces remarques et de l'égalité $H^1(V_{\bark},\Q/\Z(1))=\Hom(\pi_1^\et(V_{\bark},\bar v),\Q/\Z(1))$,
on tire un isomorphisme canonique
\begin{align}
\label{eq:pic hab dual}
\Pic(V_{\bark})=\Hom(H_{\bar v}^\ab,\bark{}^*)
\end{align}
de groupes finis munis d'une action continue de $\Gal(\bark/k)$.
Comme nous le verrons plus bas,
cet isomorphisme s'identifie,
via~\eqref{eq:identification type},
au type du torseur $G_\bark/H_{\bar v}'\to G_\bark/H_{\bar v}=V_\bark$ sous~$H_{\bar v}^\ab$,
où~$H_{\bar v}'$ désigne le sous-groupe dérivé de~$H_{\bar v}$
(lemme~\ref{lem:typerevetement} ci-dessous).

\subsection{Relèvement de l'action de \texorpdfstring{$G$}{G}}

\begin{prop}
\label{prop:relevement action}
Soit~$V$ un espace homogène de~$G$ à stabilisateur géométrique fini.
Soit $\pi:W\to V$ un revêtement étale. Si~$W$ est géométriquement irréductible,
il existe une unique action de~$G$ sur~$W$
telle que~$\pi$ soit $G$\nobreakdash-équivariant.
Munie de cette action, la variété~$W$ est un espace homogène de~$G$ à stabilisateur géométrique fini.
\end{prop}

\begin{proof}
Prouvons d'abord l'unicité.
Notons $m_V:G\times_k V \to V$ l'action de~$G$ sur~$V$.
Deux actions $m^1_W,m^2_W:G \times_k W \to W$
de~$G$ sur~$W$ rendant~$\pi$ équivariant
donnent naissance à deux morphismes
du revêtement étale $\mathrm{Id_G}\times \pi:G \times_k W \to G \times_k V$
vers le revêtement étale de $G\times_k V$ obtenu à partir de~$\pi$
par le changement de base~$m_V$.
Ces deux morphismes de revêtements étales coïncident le long de $\{1\} \times_k W$.
Comme $G \times_k W$ est connexe, ils sont donc nécessairement égaux; d'où $m^1_W=m^2_W$.

Par descente galoisienne, l'existence et l'unicité de l'action recherchée sur~$k$ résultent
de son existence et de son unicité sur~$\bark$.
Il reste donc, pour conclure, à vérifier son existence sur~$\bark$.
Avec les notations du~\textsection\ref{subsec:action exterieure},
comme $G_{\bark} \to G_{\bark}/H_{\bar v}=V_{\bark}$ est le revêtement universel de~$V_{\bark}$
et comme~$W_{\bark}$ est connexe, le choix d'un relèvement $\bar w \in W(\bark)$
de~$\bar v$
permet d'identifier le groupe $\pi_1^\et(W_\bark,\bar w)$
à un sous-groupe de $H_{\bar v}$
et le morphisme~$\pi\otimes_k\bark$
à la projection $G_{\bark}/\pi_1^\et(W_\bark,\bar w)\to G_{\bark}/H_{\bar v}$.
L'action recherchée existe donc bien sur~$\bark$.
\end{proof}

\subsection{Torseurs et revêtements}

Si~$\Rhat$ est un groupe abélien fini muni d'une action continue de $\Gal(\bark/k)$
et $\lambda:\Rhat\to\Pic(V_{\bark})$ est un homomorphisme $\Gal(\bark/k)$\nobreakdash-équivariant injectif,
on pose $R=\Homrond(\Rhat,\Gm)$
et l'on note $\widehat\lambda: H_{\bar v} \to R(\bark)$ la composée de la flèche
d'abélianisation
$H_{\bar v} \to H_{\bar v}^\ab$ et de l'homomorphisme $H_{\bar v}^\ab \to R(\bark)$
dual de~$\lambda$
compte tenu de~\eqref{eq:pic hab dual}.

\begin{prop}
\label{prop:torseurs sont eh}
Avec les notations qui précèdent,
soit $\pi:W\to V$ un torseur sous~$R$, de type~$\lambda$.
Il existe une unique action de~$G$ sur~$W$
telle que~$\pi$ soit $G$\nobreakdash-équivariant.
Munie de cette action, la variété~$W$ est un espace homogène de~$G$ à stabilisateur géométrique fini.
De plus, si $\bar w \in W(\bark)$ est tel que $\pi(\bar w)=\bar v$
et si $H_{\bar w} \subset G(\bark)$
désigne son stabilisateur,
le morphisme $\widehat\lambda:H_{\bar v}\to R(\bark)$
induit une suite exacte
\begin{align}
\label{se:hw hv r}
\xymatrix@C=1.5em{
1 \ar[r] & H_{\bar w} \ar[r] & H_{\bar v} \ar[r]^(.32){\widehat\lambda} & R(\bark) \to 1
}
\end{align}
de groupes finis.
Enfin, les actions extérieures de $\Gal(\bark/k)$
sur~$H_{\bar w}$ et~$H_{\bar v}$ sont compatibles,
au sens où
pour tout $\gamma \in \Gal(\bark/k)$,
les automorphismes extérieurs
de~$H_{\bar w}$ et~$H_{\bar v}$
déterminés par~$\gamma$
peuvent être représentés par des automorphismes~$\gamma_{\bar w}$
et~$\gamma_{\bar v}$ de ces deux groupes tels que $\gamma_{\bar w}(x)=\gamma_{\bar v}(x)$
pour tout $x \in H_{\bar w}$.
\end{prop}

\begin{proof}
Commençons par un lemme.

\begin{lem}
\label{lem:typerevetement}
Le revêtement étale
\begin{align}
\label{eq:torseur lambda geom}
G_{\bark}/\Ker\big(\mkern1mu\widehat\lambda\mkern1mu\big) \to G_{\bark}/H_{\bar v}= V_{\bark}
\end{align}
est un torseur sous~$R_{\bark}$ dont le type est~$\lambda$.
\end{lem}

\begin{proof}
Le morphisme $\widehat\lambda$ étant surjectif,
ce revêtement est bien un torseur sous~$R_{\bark}$.
Sa classe
dans $H^1(V_\bark,R_{\bark})=\Hom(\pi_1^\et(V_\bark,\bar v),R(\bark))$
est le morphisme composé
\begin{align}
\label{eq:type morphisme compose}
\pi_1^\et(V_\bark,\bar v)=H_{\bar v}\to
H_{\bar v}/\Ker\big(\mkern1mu\widehat\lambda\mkern1mu\big) \xrightarrow{\widehat\lambda}
R(\bark)\rlap{.}
\end{align}
Compte tenu des identifications
$\Pic(V_\bark)=H^1(V_\bark,\Q/\Z(1))=\Hom(\pi_1^\et(V_\bark,\bar v),\Q/\Z(1))$,
son type envoie $\chi \in \Rhat=\Hom(R(\bark),\Q/\Z(1))$
sur le morphisme obtenu en composant~\eqref{eq:type morphisme compose}
avec $\chi:R(\bark)\to \Q/\Z(1)$.
Celui-ci représente bien $\lambda(\chi)$.
\end{proof}

Démontrons maintenant la proposition.
Comme le torseur~$\pi$ est de type~$\lambda$, il devient isomorphe,
sur~$\bark$,
à~\eqref{eq:torseur lambda geom}.
En particulier, la variété~$W$ est géométriquement irréductible.
L'existence et l'unicité de l'action de~$G$
résultent alors de la
proposition~\ref{prop:relevement action}.
En outre, il apparaît sur~\eqref{eq:torseur lambda geom} que
$H_{\bar w}=\Ker\big(\mkern1mu\widehat\lambda\mkern1mu\big)$,
d'où l'exactitude de~\eqref{se:hw hv r}.
Quant à la compatibilité des actions extérieures de
$\Gal(\bark/k)$
sur~$H_{\bar w}$ et~$H_{\bar v}$, elle découle de l'existence d'un morphisme
entre les suites exactes~\eqref{se:pi1} associées à~$W$
et à~$V$.
\end{proof}

\begin{cor}
\label{cor:torseurs universels eh}
Soit~$V$ un espace homogène de~$G$ à stabilisateur géométrique fini.
Notant~$H$ le stabilisateur d'un point géométrique de~$V$,
les torseurs universels de~$V$ sont des espaces homogènes de~$G$ à stabilisateur géométrique
isomorphe au sous-groupe dérivé~$H'$.
\end{cor}

\subsection{Passage aux sous-groupes de Sylow}

La proposition suivante servira dans la preuve du théorème~\ref{th but}.

\begin{prop}
\label{prop:eh sylow}
Soit~$V$ un espace homogène de~$G$, à stabilisateur géométrique fini.
Soit~$p$ un nombre premier.
Il existe une extension finie $\ell/k$ de degré premier à~$p$,
un espace homogène~$W$, sur~$\ell$, de $G \otimes_k \ell$,
à stabilisateur géométrique fini d'ordre une
puissance de~$p$,
et un morphisme fini étale $W \to V \otimes_k \ell$ de degré premier à~$p$.
\end{prop}

\begin{proof}
Commençons par un lemme de théorie des groupes.

\begin{lem}
\label{lem:sylow}
Soit~$p$ un nombre premier.  Toute suite exacte
\begin{align*}
\xymatrix{
1 \ar[r] & H \ar[r] & E \ar[r] & \Gamma \ar[r] & 1
}
\end{align*}
de groupes profinis, avec~$H$ fini,
s'inscrit dans un diagramme commutatif à lignes exactes
\begin{align*}
\xymatrix@R=2ex{
1 \ar[r] & \vphantom{H_.}\smash[b]{H_p} \ar[r]
\ar@{}[d]|-*{\cap}
& \vphantom{E_.}\smash[b]{E_p} \ar[r]
\ar@{}[d]|-*{\cap}
& \vphantom{\Gamma_.}\smash[b]{\Gammap}
\ar@{}[d]|-*{\cap}\ar[r] & 1 \\
1 \ar[r] & H \ar[r] & E \ar[r] & \Gamma \ar[r] & 1\rlap{\text{,}}
}
\end{align*}
où $H_p$, $E_p$, $\Gammap$ sont des $p$\nobreakdash-Sylow de~$H$, $E$, $\Gamma$
respectivement.
\end{lem}

\begin{proof}
Soit~$E_p$ un $p$\nobreakdash-Sylow de~$E$
(voir \cite[Chapitre~I, \textsection1.4, Proposition~3]{serrecg}).
D'après \emph{loc.\ cit.}, Proposition~4~(b),
l'image~$\Gammap$ de~$E_p$ dans~$\Gamma$
est un $p$\nobreakdash-Sylow de~$\Gamma$.
Posons $H_p=H \cap E_p$
et notons $E_0 \subset E$ l'image réciproque de~$\Gammap$.
En termes de nombres surnaturels (\emph{loc.\ cit.}, \textsection1.3),
l'ordre de~$E_0$, extension de~$\Gammap$ par~$H$,
est le produit d'une puissance surnaturelle de~$p$ et
d'un entier fini premier à~$p$.
L'ensemble $E_0/E_p$ est donc fini d'ordre premier à~$p$.
Comme $H/H_p=E_0/E_p$, il s'ensuit que~$H_p$ est
un $p$\nobreakdash-Sylow de~$H$.
\end{proof}

Notons $H_p\subset \pi_1^\et(V_\bark,\bar v)$, $E_p\subset \pi_1^\et(V,\bar v)$,
$\Gammap \subset \Gal(\bark/k)$
les $p$\nobreakdash-Sylow donnés par le lemme appliqué à la suite exacte~\eqref{se:pi1}.
Notons $k'\subset \bark$ le sous-corps des invariants de~$\Gammap$.
L'image réciproque de $\Gammap=\Gal(\bark/k')\subset\Gal(\bark/k)$ dans~$\pi_1^\et(V,\bar v)$ s'identifie
à $\pi_1^\et(V\otimes_kk',\bar v)$, d'où un diagramme commutatif à lignes exactes
\myxyin
\begin{align*}
\xymatrix@R=2ex{
1 \ar[r] & \vphantom{H_.}\smash[b]{H_p} \ar[r]
\ar@{}[d]|-*{\cap}
& \vphantom{E_.}\smash[b]{E_p} \ar[r]
\ar@{}[d]|-*{\cap}
& \vphantom{\Gamma_.\pi_1}\smash[b]{\Gammap}
\save+<0ex,-.95em>\ar@{=}[d]!(0,3.3)\restore
\ar[r] & 1 \\
1 \ar[r] & \vphantom{\pi_1}\smash[t]{\pi_1^\et}(V_{\bark},\bar v)
\save+<0ex,-.95em>\ar@{=}[d]!(0,3.6)\restore
 \ar[r] & \vphantom{\pi_1}\smash[t]{\pi_1^\et}(V\otimes_kk',\bar v)
\ar@{}[d]|-*{\cap}
 \ar[r] & \Gammap
\ar@{}[d]|-*{\cap}
 \ar[r] & 1\\
1 \ar[r] & \vphantom{\pi_1}\smash[t]{\pi_1^\et}(V_{\bark},\bar v) \ar[r] & \vphantom{\pi_1}\smash[t]{\pi_1^\et}(V,\bar v) \ar[r] & \Gal(\bark/k) \ar[r] & 1\rlap{.}
}
\end{align*}
\myxyout
On voit sur ce diagramme que l'application canonique
$\pi_1^\et(V_\bark,\bar v)/H_p \to \pi_1^\et(V\otimes_kk',\bar v)/E_p$
est une bijection; ainsi~$E_p$ est-il un sous-groupe d'indice fini premier à~$p$
de $\pi_1^\et(V\otimes_kk',\bar v)$.
Il~lui correspond un revêtement étale pointé de $V\otimes_k k'$, géométriquement connexe sur~$k'$
et de degré premier à~$p$.
Choisissons une sous-extension finie $\ell/k$ de~$k'/k$
telle que ce revêtement provienne, par extension des scalaires, d'un revêtement étale $W \to V \otimes_k\ell$
géométriquement connexe sur~$\ell$, muni d'un point $\bar w \in W(\bark)$.
Comme~$\Gammap$ est un $p$\nobreakdash-Sylow de $\Gal(\bark/k)$, le degré de~$\ell$ sur~$k$
est premier à~$p$.
D'après la proposition~\ref{prop:relevement action}, la variété~$W$ est un espace homogène
de $G \otimes_k\ell$ à stabilisateur géométrique fini.
Ce stabilisateur s'identifie à $\pi_1^\et(W\otimes_\ell \bark,\bar w)=H_p$, dont l'ordre
est bien une puissance de~$p$.
\end{proof}

\section{Points rationnels des espaces homogènes}
\label{sec:points rat}

Nous combinons, dans ce paragraphe, le contenu des \textsectiondouble\ref{sec:descente}--\ref{sec:pi1eh}
afin d'en déduire des résultats sur les points rationnels des espaces homogènes de groupes linéaires
semi-simples simplement connexes, à stabilisateur géométrique fini, sur les corps de nombres.

\subsection{Un énoncé conditionnel}
\label{subsec:enonce conditionnel}

Afin de présenter la stratégie générale employée dans tout l'article,
on commence par le théorème conditionnel suivant.
Nous rendrons cette stratégie inconditionnelle dans deux cas: celui où le stabilisateur
géométrique est abélien, retrouvant ainsi un théorème de Borovoi~\cite{borovoi},
et celui où le stabilisateur géométrique est \emph{hyper-résoluble} en tant que groupe fini muni d'une action
extérieure de $\Gal(\bark/k)$.

\begin{thm}
\label{th:pointsrationnelsconjectural}
Soit~$X$ une variété propre, lisse et géométriquement irréductible sur
un corps de nombres~$k$.
Supposons~$X$
birationnellement équivalente à un espace homogène d'un groupe
algébrique linéaire semi-simple simplement connexe, à stabilisateur géométrique
fini résoluble.
Si la conjecture \cite[Conjecture~9.1]{hw} est vraie pour~$k$,
l'ensemble $X(k)$ est dense dans $X(\A_k)^{\Br(X)}$.
\end{thm}

\begin{proof}
Notons~$V$ l'espace homogène et~$G$ le groupe algébrique linéaire semi-simple simplement connexe.
Quitte à remplacer~$X$ par une compactification lisse de~$V$, on peut supposer que~$V$
est un ouvert dense de~$X$ (voir \cite[Proposition~6.1~(iii)]{cps}, \cite[Lemma~1.3]{ctskogoodreduction}).
Fixons $\bar v \in V(\bark)$,
notons
$H_{\bar v} \subset G(\bark)$
son stabilisateur
et prouvons le théorème par récurrence sur l'ordre de~$H_{\bar v}$.
Lorsque $H_{\bar v}$ est trivial,
l'ensemble~$X(k)$ est même dense dans~$X(\A_k)$
(Kneser, Harder, Chernousov; voir \cite[Theorem~6.6, Theorem~7.8]{platonovrapinchuk}).
Pour~$H_{\bar v}$ fini résoluble quelconque,
supposons $X(\A_k)^{\Br(X)}\neq\emptyset$.
La variété~$X$ admet alors un torseur universel
(voir \cite[Proposition~6.1.4]{skobook}).
Le théorème~\ref{th descente} appliqué
à un tel torseur
montre que pour établir la densité de $X(k)$ dans $X(\A_k)^{\Br(X)}$,
il suffit d'établir celle
de~$Z(k)$ dans~$Z(\A_k)^{\Br(Z)}$
pour une compactification lisse~$Z$ de chaque torseur universel de~$X$.
Fixons~$Z$.
D'après le corollaire~\ref{cor:torseurs universels},
il existe
un ouvert dense $Z^0\subset Z$,
un tore quasi-trivial~$Q$ et un morphisme
lisse $f^0:Z^0\to Q$ dont les fibres sont des torseurs universels de~$V$.
Par le corollaire~\ref{cor:torseurs universels eh},
ceux-ci
sont des espaces homogènes de~$G$ à stabilisateur géométrique
fini résoluble d'ordre strictement inférieur à l'ordre de~$H_{\bar v}$, puisque~$H_{\bar v}$ est résoluble.
Grâce à l'hypothèse de récurrence, on peut donc appliquer
le théorème~\ref{thm:fibration cas general pointsrat} à une fibration $f:Z'\to \P^n_k$
compactifiant~$f^0$ et conclure que $Z'(k)$ est dense dans $Z'(\A_k)^{\Br(Z')}$
et donc que $Z(k)$ est dense dans $Z(\A_k)^{\Br(Z)}$.
\end{proof}

\subsection{Stabilisateurs finis abéliens}

Le théorème suivant est dû à Borovoi~\cite{borovoi}.

\begin{thm}
Soit~$X$ une variété propre, lisse et géométriquement irréductible sur
un corps de nombres~$k$.
Supposons~$X$
birationnellement équivalente à un espace homogène d'un groupe
algébrique linéaire semi-simple simplement connexe, à stabilisateur géométrique
fini abélien.
L'ensemble $X(k)$ est dense dans $X(\A_k)^{\Br_1(X)}$.
\end{thm}

Le point de vue adopté ici le démontre aussi:
il suffit de suivre la preuve du théorème~\ref{th:pointsrationnelsconjectural}
en remplaçant
le théorème~\ref{th descente}
par
\cite[Theorem~6.1.1]{skobook}
et, à la fin, le
théorème~\ref{thm:fibration cas general pointsrat}
par le
théorème~\ref{th:fibration skohar}~(i),
dont l'hypothèse sur les fibres au-dessus des points de codimension~$1$
est satisfaite grâce à
\cite[Théorème~4.2]{ctk}.

Alternativement, on peut éviter le recours à \cite{ctk} ainsi
qu'au théorème~\ref{th:fibration skohar}~(i) en remarquant,
suivant Borovoi~\cite[\textsection3]{borovoi},
qu'une version
 élémentaire de la méthode des fibrations s'applique ici:
en effet, les fibres
lisses de la fibration considérée
sont des torseurs
sous des groupes algébriques linéaires connexes semi-simples simplement connexes,
donc vérifient les hypothèses du lemme ci-dessous
d'après \cite[Theorem~6.4, Theorem~6.6, Theorem~7.8]{platonovrapinchuk}.
Ce lemme nous resservira au~\textsection\ref{sec:zerocycles}.

\begin{lem}
\label{lem:fibration cas trivial}
Soit $f:Z\to B$ un morphisme dominant
entre variétés irréductibles lisses sur un corps de nombres~$k$.
Considérons, pour une variété~$Y$ sur~$k$, la propriété suivante:
\vspace*{1.5pt}
\begin{flushright}
$(\star)\mkern9mu$\begin{minipage}[t]{.92\textwidth}
si $Y(k_v)\neq\emptyset$ pour toute place réelle~$v$,
alors $Y(k)\neq\emptyset$ et $Y$ vérifie l'approximation faible.
\end{minipage}
\end{flushright}
\vspace*{3pt}
Si~$B$ et les fibres de~$f$ au-dessus des points rationnels d'un ouvert dense de~$B$
vérifient~$(\star)$,
alors~$Z$ vérifie~$(\star)$ aussi.
\end{lem}

Notons que si~$Y$ est une variété lisse et $Y^0 \subset Y$ est un ouvert dense,
la propriété~$(\star)$ vaut pour~$Y$ si et seulement si elle vaut pour~$Y^0$,
en vertu du théorème des fonctions implicites.  La propriété~$(\star)$ est donc
un invariant birationnel des variétés lisses sur~$k$.

\begin{proof}[Démonstration du lemme~\ref{lem:fibration cas trivial}]
Compte tenu de la remarque qui précède, on peut, quitte à rétrécir~$Z$ et~$B$, supposer que~$f$
est lisse et que les fibres de~$f$ au-dessus de tous les points rationnels de~$B$ vérifient~$(\star)$.
Supposons que $Z(k_v)\neq\emptyset$ pour toute place~$v$ réelle.
Soit~$S$ un ensemble fini de places de~$k$ contenant les places réelles;
on prend~$S$ égal à l'ensemble des places réelles pour prouver que $Z(k)\neq\emptyset$,
arbitraire pour établir l'approximation faible pour~$Z$.
Pour chaque $v \in S$, soit $P_v \in Z(k_v)$.
D'après~$(\star)$ pour~$B$, il existe $b \in B(k)$ arbitrairement proche des~$f(P_v)$.
Quitte à modifier les~$P_v$, on peut supposer, grâce au théorème des fonctions implicites,
que $f(P_v)=b$ pour $v \in S$.
Par la propriété~$(\star)$ pour~$Z_b$, il s'ensuit que~$Z_b$ (et donc~$Z$) contient un point rationnel
arbitrairement proche des~$P_v$.
\end{proof}

\subsection{Stabilisateurs finis hyper-résolubles}

La notion de groupe hyper-résoluble
(voir \cite[\textsection8.3]{serrereplin}) s'étend aux groupes munis d'une action extérieure
de $\Gal(\bark/k)$:

\begin{defn}
\label{def:hyperresoluble}
Un groupe fini~$H$ muni d'une action extérieure de $\Gal(\bark/k)$
est dit \emph{hyper-résoluble} s'il existe un entier~$m$ et une suite
\begin{align*}
\{1\} = H_0 \subset H_1 \subset \dots \subset H_m = H
\end{align*}
de sous-groupes distingués de~$H$, stables sous l'action extérieure de~$\Gal(\bark/k)$, avec $H_i/H_{i-1}$ cyclique pour
$i\in\{1,\dots,m\}$.  On parle ici de stabilité d'un sous-groupe distingué sous un automorphisme extérieur
pour désigner sa stabilité sous un automorphisme quelconque représentant l'automorphisme extérieur en question;
cette propriété ne dépend pas du représentant choisi puisque le sous-groupe est distingué.
\end{defn}

\begin{exemple}
Lorsque l'action extérieure est triviale, cette définition coïncide
avec la définition classique.  En particulier, 
munis de l'action extérieure triviale,
les groupes nilpotents
et les groupes diédraux sont hyper-résolubles
(voir \cite[\textsection8.3]{serrereplin}).
\end{exemple}

Lorsque le stabilisateur géométrique est hyper-résoluble en tant que groupe fini muni d'une
action extérieure de $\Gal(\bark/k)$ (voir~\textsection\ref{subsec:action exterieure}),
la preuve du théorème~\ref{th:pointsrationnelsconjectural} peut être rendue inconditionnelle
grâce au théorème~\ref{thm:fibration presque favorable}~(ii) et à la proposition~\ref{prop:sectionbark}~(ii).

\begin{thm}
\label{th:66}
Soit~$X$ une variété propre, lisse et géométriquement irréductible sur
un corps de nombres~$k$.
Si~$X$ est
birationnellement équivalente à un espace homogène d'un groupe
algébrique linéaire semi-simple simplement connexe, à stabilisateur géométrique
fini hyper-résoluble (au sens de la définition~\ref{def:hyperresoluble}),
alors $X(k)$ est dense dans $X(\A_k)^{\Br(X)}$.
\end{thm}

\begin{proof}
Notons~$V$ l'espace homogène et~$G$ le groupe algébrique linéaire semi-simple simplement connexe,
fixons $\bar v \in V(\bark)$
et notons
$H_{\bar v} \subset G(\bark)$ son stabilisateur.
Comme dans la preuve du théorème~\ref{th:pointsrationnelsconjectural},
on peut supposer que~$V$ est un ouvert de~$X$;
comme dans la preuve du théorème~\ref{th:pointsrationnelsconjectural},
le résultat recherché est vrai
lorsque $H_{\bar v}$ est trivial
et l'on procède en général par récurrence
sur l'ordre de~$H_{\bar v}$ pour l'établir.

Soit
$\{1\} = H_0 \subset H_1 \subset \dots \subset H_m = H_{\bar v}$ une suite de sous-groupes
remplissant les conditions de la définition~\ref{def:hyperresoluble},
avec $H_{m-1}\neq H_m$.
L'action extérieure de $\Gal(\bark/k)$ sur~$H_{\bar v}$ induit une action de $\Gal(\bark/k)$
sur le groupe cyclique $H_{\bar v}/H_{m-1}$ et donc aussi sur le groupe cyclique
$\Rhat=\Hom(H_{\bar v}/H_{m-1},\bark{}^*)$.
Rappelons que
$\Pic(V_{\bark})=\Hom(H_{\bar v}^{\ab},\bark{}^*)$
(voir~\textsection\ref{subsec:stab picard}).
Soit $\lambda:\Rhat\hookrightarrow \Pic(V_{\bark})$
l'injection duale de la projection $H_{\bar v}^\ab \twoheadrightarrow H_{\bar v}/H_{m-1}$.

Nous sommes maintenant dans la situation du~\textsection\ref{sec:duntorseuralautre},
dont nous reprenons les notations~$R$, $T$, $Q$, $\nu$ (voir \eqref{diag:qtr}, \eqref{eq:rtq}).
L'application $\nu:\widehat T\to \Pic(X_{\bark})$ est injective puisque~$\lambda$ l'est.
Or $\Pic(X_{\bark})$ est sans torsion;
par conséquent~$\widehat T$ est sans torsion, autrement dit~$T$ est un tore.
Compte tenu qu'un torseur
sur~$X$, sous~$T$, de type~$\nu$ existe dès que $X(\A_k)^{\Br(X)}\neq\emptyset$
(voir \cite[Theorem~6.1.1]{skobook}),
le théorème~\ref{th descente} appliqué à un tel torseur
montre que pour que~$X(k)$ soit dense dans $X(\A_k)^{\Br(X)}$,
il suffit
que $Z(k)$ soit dense dans $Z(\A_k)^{\Br(Z)}$
pour toute compactification lisse~$Z$ d'un torseur sur~$X$, sous~$T$, de type~$\nu$.

D'après la proposition~\ref{prop:comparaisontorseurs},
il existe
un ouvert dense $Z^0\subset Z$
et un morphisme
lisse $f^0:Z^0\to Q$ dont les fibres sont des torseurs sur~$V$, sous~$R$, de type~$\lambda$.
Vu la première partie de la
proposition~\ref{prop:torseurs sont eh}, les fibres de~$f^0$ sont
des espaces
homogènes de~$G$ à stabilisateur géométrique fini.
En particulier, les compactifications lisses de la fibre générique de~$f^0$ sont rationnellement connexes
et comme~$\Rhat$ est cyclique, la proposition~\ref{prop:sectionbark}~(ii) garantit que
$f^0\otimes_k\bark$ admet une section rationnelle.

Il résulte de la seconde partie de la
proposition~\ref{prop:torseurs sont eh}
que les fibres de~$f^0$ au-dessus des points rationnels de~$Q$
sont des espaces homogènes dont le stabilisateur d'un point géométrique au-dessus de~$\bar v$ est
égal, en tant que sous-groupe de~$H_{\bar v}$,
à~$H_{m-1}$.
Ces stabilisateurs viennent avec une action extérieure de $\Gal(\bark/k)$
sous laquelle la suite $H_0 \subset H_1 \subset \cdots \subset H_{m-1}$
de sous-groupes distingués de~$H_{m-1}$ est stable,
en vertu de la troisième partie de la proposition~\ref{prop:torseurs sont eh}
et de la stabilité des~$H_i$ sous l'action extérieure de $\Gal(\bark/k)$
sur~$H_{\bar v}$.
Les fibres de~$f^0$ au-dessus des points rationnels de~$Q$
sont donc des espaces homogènes de~$G$ à stabilisateur géométrique
fini hyper-résoluble
au sens de la définition~\ref{def:hyperresoluble},
d'ordre strictement inférieur à celui de~$H_{\bar v}$.
L'hypothèse de récurrence permet maintenant d'appliquer
le théorème~\ref{thm:fibration presque favorable}~(ii)
à une fibration $f:Z'\to \P^n_k$
compactifiant~$f^0$.
En effet, les fibres de $f\otimes_k \bark$ au-dessus de tous les points de codimension~$1$
sont scindées puisque $f\otimes_k\bark$ admet une section rationnelle.
On conclut que $Z'(k)$ est dense dans $Z'(\A_k)^{\Br(Z')}$
et donc que $Z(k)$ est dense dans $Z(\A_k)^{\Br(Z)}$.
\end{proof}

\section{Zéro-cycles sur les espaces homogènes}
\label{sec:zerocycles}

Nous établissons ici le théorème~\ref{th but}.  Notons dorénavant~$k$ un corps de nombres.

\subsection{Réduction aux stabilisateurs finis}
\label{subsec:reductionstabfini}

Demarche et Lucchini Arteche \cite{demarchelucchinireduction} ont démontré qu'une
\emph{bonne} propriété des espaces homogènes
au sens de la définition rappelée ci-dessous est vérifiée par les espaces homogènes de groupes algébriques
linéaires connexes si elle est vérifiée par les espaces homogènes de~$\SL_n$ à stabilisateur
géométrique fini.

\begin{defn}
\label{def:bonne propriete}
Une propriété~$P$
des espaces homogènes de groupes algébriques linéaires connexes sur~$k$
est \emph{bonne}
si elle remplit les conditions suivantes:
\begin{enumerate}
\item[(i)] si~$V$ et~$W$ sont stablement birationnellement équivalents sur~$k$, alors $P(V) \Leftrightarrow P(W)$;
\item[(ii)] si $V\to W$ est un morphisme admettant une section, alors $P(V)\Rightarrow P(W)$;
\item[(iii)] si $V\to W$ est un morphisme surjectif dont les fibres sont des espaces homogènes de
    groupes algébriques linéaires connexes semi-simples simplement connexes à stabilisateur géométrique
extension d'un groupe algébrique linéaire connexe semi-simple par un groupe algébrique linéaire connexe unipotent,
alors $P(W) \Rightarrow P(V)$.
\end{enumerate}
Ci-dessus,
les symboles~$V$ et~$W$ désignent des espaces homogènes de groupes algébriques linéaires connexes sur~$k$
et le terme \og{}morphisme\fg{} signifie \og{}morphisme de variétés\fg{}.
\end{defn}

Nous ne savons pas si la validité de la conjecture~$(E)$ pour les compactifications lisses d'un espace homogène
est une bonne propriété des espaces homogènes sur~$k$.
La variante suivante de la conjecture~$(E)$
se comporte mieux vis-à-vis de la définition~\ref{def:bonne propriete}.

\begin{defn}
\label{def:eplus}
Une variété~$X$ irréductible, propre et lisse sur~$k$ \emph{vérifie la propriété~$(E^+)$} si
toute variété~$Z$ irréductible, propre et lisse sur~$k$ munie d'un morphisme dominant $Z \to X$
remplissant les deux conditions suivantes
vérifie la conjecture~$(E)$:
\begin{enumerate}
\item[(i)]
la fibre générique de ce morphisme est rationnellement connexe;
\item[(ii)]
il existe un ouvert dense $X^0 \subset X$ tel que
pour toute extension finie~$k'/k$ et tout $x \in X^0(k')$,
la fibre~$Z_x$
vérifie la condition~$(\star)$ du lemme~\ref{lem:fibration cas trivial}
sur le corps de nombres~$k'$.
\end{enumerate}
\end{defn}

\begin{prop}
\label{prop:eplus bonne propriete}
La propriété \og{}toute variété propre, lisse
et birationnellement équivalente à~$V$ vérifie
la propriété~$(E^+)$\fg{} est une bonne propriété des espaces homogènes de groupes algébriques linéaires
connexes~$V$ sur~$k$ au sens de la définition~\ref{def:bonne propriete}.
\end{prop}

\begin{proof}
La condition~(i) de la définition~\ref{def:bonne propriete} résulte du lemme~\ref{lem:eplus invariant birationnel stable}
ci-dessous et la condition~(ii) résulte du lemme~\ref{lem:e avec section se transmet}~(ii).
La condition~(iii) vient quant à elle du lemme~\ref{lem:eplus se releve},
compte tenu que d'après~\cite[Proposition~3.4]{borovoi},
tout morphisme $Y \to X$ entre variétés irréductibles, propres et lisses sur~$k$
compactifiant le morphisme $V\to W$ qui apparaît dans~(iii) satisfait aux conditions~(i) et~(ii)
de la définition~\ref{def:eplus}.
\end{proof}

\begin{lem}
\label{lem:eplus se releve}
Soit $Y \to X$ un morphisme dominant entre variétés irréductibles, propres et lisses sur~$k$,
remplissant les conditions~(i) et~(ii) de la définition~\ref{def:eplus}.
Si la propriété~$(E^+)$ vaut pour~$X$, elle vaut pour~$Y$ aussi.
\end{lem}

\begin{proof}
En effet, les conditions~(i) et~(ii) de la définition~\ref{def:eplus} sont
stables par composition de morphismes dominants,
au vu du théorème de Graber--Harris--Starr \cite[Theorem~1.1]{ghs}
et du lemme~\ref{lem:fibration cas trivial}.
\end{proof}

\begin{lem}
\label{lem:eplus invariant birationnel stable}
La propriété~$(E^+)$ est un invariant birationnel stable des variétés irréductibles, propres et lisses sur~$k$.
\end{lem}

\begin{proof}
Soient~$X$ et~$Y$ deux variétés irréductibles, propres et lisses sur~$k$
et $Y \dashrightarrow X$
une application birationnelle.
Si~$Z$ est une variété irréductible, propre et lisse sur~$k$ munie d'un morphisme dominant $Z \to Y$
remplissant les conditions~(i) et~(ii) de
la définition~\ref{def:eplus}, résolvons, avec Hironaka~\cite{hironaka},
l'indétermination de l'application composée $Z \dashrightarrow X$. On obtient ainsi une variété~$Z'$ irréductible, propre et lisse
sur~$k$ et
un morphisme birationnel $Z' \to Z$ tel que l'application rationnelle composée $Z' \dashrightarrow X$
soit un morphisme.
Les fibres génériques de $Z\to Y$ et de $Z' \to X$ étant birationnellement équivalentes,
l'invariance birationnelle de la propriété~$(\star)$ parmi les variétés lisses montre que le morphisme $Z' \to X$ remplit
lui aussi les conditions~(i) et~(ii) de la définition~\ref{def:eplus}.
La propriété~$(E^+)$ pour~$X$ implique donc la propriété~$(E)$ pour~$Z'$,
qui implique à son tour la propriété~$(E)$ pour~$Z$, par invariance birationnelle de~$(E)$
(voir \cite[Remarque~1.1~(vi)]{wittdmj}),
d'où, en faisant varier~$Z$, la propriété~$(E^+)$ pour~$Y$.
Il s'ensuit
que~$(E^+)$ est un invariant birationnel.

Vérifions maintenant que~$(E^+)$ pour~$X$ équivaut à~$(E^+)$ pour $X \times_k \P^1_k$.
Si~$X$ satisfait~$(E^+)$, alors $X \times_k \P^1_k$ aussi, par le lemme~\ref{lem:eplus se releve}.
Réciproquement, si $X\times_k \P^1_k$ satisfait~$(E^+)$ et si $Z \to X$ est
comme dans la définition~\ref{def:eplus}, le morphisme
$Z \times_k \P^1_k\to X \times_k \P^1_k$ vérifie les conditions~(i) et~(ii) de cette
définition, donc la conjecture~$(E)$ vaut pour $Z \times_k \P^1_k$, donc
pour~$Z$ puisque le groupe de Brauer et le groupe de Chow des zéro-cycles
sont des invariants stables.
\end{proof}

\begin{lem}
\label{lem:e avec section se transmet}
Soit $f:Y \to X$ un morphisme entre variétés irréductibles, propres et lisses sur~$k$.
Supposons que~$f$ admette une section rationnelle.
\begin{enumerate}
\item[(i)]
Si la propriété~$(E)$ vaut pour~$Y$, elle vaut pour~$X$ aussi.
\item[(ii)]
Si la propriété~$(E^+)$ vaut pour~$Y$, elle vaut pour~$X$ aussi.
\end{enumerate}
\end{lem}

\begin{proof}
Pour toute variété~$Z$ propre et lisse sur~$k$, notons~$h_Z$
le groupe de cohomologie du complexe~\eqref{se:egeneral}
associé à~$Z$, de sorte que~$(E)$ vaut pour~$Z$ si et seulement si $h_Z=0$.
Le morphisme~$f$
induit un homomorphisme $f_*:h_Y\to h_X$.
Choisissons une section rationnelle de~$f$ et résolvons-en l'indétermination
grâce à Hironaka~\cite{hironaka}: on obtient
une variété propre et lisse~$X'$ sur~$k$,
un morphisme birationnel $\pi:X'\to X$ et un morphisme $s:X'\to Y$ tels que $f\circ s=\pi$.
Comme~$\pi$ est birationnel, le morphisme $\pi_*:h_{X'}\to h_X$ est un isomorphisme
(voir \cite[Remarque~1.1~(vi)]{wittdmj}).
L'égalité $f_*\circ s_*=\pi_*$ implique alors que $f_*:h_Y\to h_X$ est surjective;
ainsi $h_Y=0$ entraîne que $h_X=0$, ce qui prouve~(i).

Pour vérifier l'assertion~(ii),
on fixe $Z \to X$ comme dans la définition~\ref{def:eplus},
on choisit une désingularisation~$Z'$ de la composante irréductible de $Z \times_X Y$
qui domine~$X$ et l'on applique~(i) au morphisme $Z'\to Z$.
\end{proof}

La proposition~\ref{prop:eplus bonne propriete}
nous permet d'appliquer \cite[Théorème~5.2]{demarchelucchinireduction}
à la propriété~$(E^+)$.
Ainsi, pour établir
le théorème~\ref{th but}, il nous reste seulement à montrer que
les compactifications lisses d'espaces homogènes de~$\SL_n$ à stabilisateur
géométrique fini vérifient~$(E^+)$.

\subsection{Réduction aux stabilisateurs finis résolubles}
\label{subsec:reductionstabfinires}

Réduisons-nous maintenant aux espaces homogènes dont le stabilisateur géométrique est
fini d'ordre une puissance d'un nombre premier et est donc \emph{a~fortiori} un groupe fini résoluble.

\begin{prop}
\label{prop:eplus en p}
Soit~$X$ une variété irréductible, propre et lisse sur~$k$.
Supposons que pour tout nombre premier~$p$, il existe une variété irréductible, propre et lisse~$Y$
sur~$k$ vérifiant la propriété~$(E^+)$ et un morphisme $Y \to X$ génériquement fini de degré premier à~$p$.
Alors~$X$ vérifie la propriété~$(E^+)$.
\end{prop}

\begin{proof}
Soit~$Z$ une variété irréductible, propre et lisse sur~$k$ munie d'un morphisme dominant $Z\to X$
vérifiant les conditions de la définition~\ref{def:eplus}.
Soit~$p$ un nombre premier.
Soit $Y\to X$ un morphisme comme dans l'énoncé de la proposition.
Soit~$Z'$ une désingularisation de la composante irréductible de $Z \times_X Y$
qui domine~$Y$.
La variété~$Z'$ vérifie~$(E)$ puisque~$Y$ vérifie~$(E^+)$.
Reprenant la notation introduite au début de la preuve du lemme~\ref{lem:e avec section se transmet},
les fonctorialités covariante et contravariante du groupe de Chow et du groupe de Brauer
(voir \cite[Chapter~1, Chapter~8]{fulton}, \cite[\textsection1.1]{ctsd94})
permettent d'en déduire que~$h_Z$ est annulé par
le degré du morphisme génériquement fini $Z' \to Z$,
qui est un entier premier à~$p$.
Ainsi, ce groupe
est un groupe abélien de torsion première à~$p$, pour tout~$p$;
il est donc nul.
\end{proof}

La proposition~\ref{prop:eh sylow}
combinée à la proposition~\ref{prop:eplus en p} ci-dessus montre
que
la propriété~$(E^+)$,
sur toutes les extensions finies de~$k$,
pour les compactifications
lisses d'espaces homogènes de~$\SL_n$ à stabilisateur géométrique fini d'ordre une puissance
d'un nombre premier
implique la propriété~$(E^+)$, sur~$k$,
 pour les compactifications lisses d'espaces homogènes de~$\SL_n$ à stabilisateur
géométrique fini.

\subsection{Passage aux torseurs universels}

Nous appliquons ici la méthode de la descente à l'étude des zéro-cycles
sur les variétés rationnellement connexes
et vérifions plus précisément
qu'afin d'établir la propriété~$(E^+)$
pour une telle variété,
il est loisible,
quitte à effectuer une extension des scalaires,
de passer à ses torseurs universels.

\begin{prop}
\label{prop:reductiontorseursuniversels}
Soit~$X$ une variété propre, lisse et rationnellement connexe, sur~$k$.
Si
pour toute extension finie $k'/k$, la propriété~$(E^+)$ vaut
pour les compactifications lisses des torseurs universels de
la variété $X\otimes_k k'$ sur~$k'$,
elle vaut pour~$X$ aussi.
\end{prop}

\begin{proof}
Fixons une variété irréductible, propre et lisse~$Z$ sur~$k$ munie d'un morphisme dominant $Z\to X$
vérifiant les conditions de la définition~\ref{def:eplus}.
Comme~$X$ et la fibre générique de ce morphisme sont rationnellement connexes,
la variété~$Z$ est rationnellement connexe (voir \cite[Theorem~1.1]{ghs}).
D'après \cite[Theorem~8.3~(3)]{hw} appliqué à la fibration triviale $Z\times_k\P^1_k \to \P^1_k$,
il suffit donc, pour vérifier la conjecture~$(E)$ pour~$Z$, de prouver que
pour toute extension finie $k'/k$,
l'image de $Z(\A_{k'})^{\Br(Z\otimes_k k')}$ dans $\CHzAhat(Z\otimes_k k')$
est contenue dans l'image de $\CHzhat(Z\otimes_k k')$ dans ce même groupe.

Quitte à remplacer~$k$ par~$k'$, on peut supposer que $k=k'$.
Fixons $(P_v)_{v\in\Omega}\in Z(\A_k)^{\Br(Z)}$.
Comme $Z(\A_k)^{\Br(Z)}\neq\emptyset$, on a $X(\A_k)^{\Br(X)}\neq\emptyset$,
par conséquent~$X$ admet un torseur universel (voir \cite[Proposition~6.1.4]{skobook}).
D'après le théorème~\ref{th descente}
appliqué au torseur sur~$Z$ obtenu par changement de base,
il existe un torseur universel $Y \to X$
et un point adélique
$(P'_v)_{v\in\Omega} \in (Z \times_X Y)(\A_k)^{\Br_{\nr}(Z\times_X Y)}$
relevant $(P_v)_{v\in\Omega}$.

Soit~$Y'$ une compactification lisse de~$Y$.
Soit~$Z'$ une compactification lisse de $Z \times_X Y$ telle que les projections
$Z\times_X Y \to Z$ et $Z \times_X Y \to Y$ s'étendent en des morphismes $p:Z'\to Z$ et $q:Z' \to Y'$.
Le morphisme~$q$ remplit les conditions~(i) et~(ii) de la définition~\ref{def:eplus}.
Comme par hypothèse la variété~$Y'$ vérifie la propriété~$(E^+)$,
la variété~$Z'$ vérifie donc~$(E)$: l'image de $(P'_v)_{v\in\Omega}$
dans $\CHzAhat(Z')$ appartient à l'image de $\CHzhat(Z')$ dans ce même groupe.
En appliquant~$p_*$, on conclut que l'image de $(P_v)_{v\in\Omega}$
dans $\CHzAhat(Z)$ provient de $\CHzhat(Z)$.
\end{proof}

\begin{rmk}
La proposition~\ref{prop:reductiontorseursuniversels} resterait vraie
(avec la même preuve, qui d'ailleurs se simplifierait quelque peu)
si dans son énoncé on remplaçait~$(E^+)$ par~$(E)$.
\end{rmk}

\subsection{Compatibilité de la propriété \texorpdfstring{$(E^+)$}{(E⁺)} aux fibrations}

\begin{prop}
\label{prop:compatibilite eplus fibrations}
Soient~$X$ une variété irréductible, propre et lisse sur~$k$
et $f:X\to \P^n_k$ un morphisme dominant de fibre générique rationnellement connexe,
pour un entier $n\geq 1$.
Si la propriété~$(E^+)$ vaut pour les fibres de~$f$ au-dessus des points fermés d'un ouvert dense de~$\P^n_k$,
elle vaut pour~$X$ aussi.
\end{prop}

\begin{proof}
Soient $Z\to X$ et $X^0\subset X$ comme dans la définition~\ref{def:eplus}.
Notons $g:Z \to \P^n_k$ la composée de ce morphisme avec~$f$.
Quitte à rétrécir~$X^0$,
on peut supposer $Z \to X$ lisse au-dessus de~$X^0$.
Soit $U \subset \P^n_k$ un ouvert dense au-dessus duquel~$f$ et~$g$ sont lisses
et au-dessus des points fermés
duquel les fibres de~$f$ vérifient~$(E^+)$.
Quitte à rétrécir~$U$, on peut supposer
l'ouvert $X^0 \cap f^{-1}(u)$ dense dans~$f^{-1}(u)$
pour tout $u \in U$.
Les fibres de~$g$ au-dessus des points fermés de~$U$ vérifient alors la conjecture~$(E)$.
D'autre part, la fibre générique de~$g$
est rationnellement connexe, d'après Graber--Harris--Starr \cite[Theorem~1.1]{ghs}.
Que~$Z$ vérifie la conjecture~$(E)$
résulte donc du
théorème~\ref{thm:fibration cas general zerocycles}
appliqué à~$g$.
\end{proof}

\subsection{Fin de la preuve du théorème \ref{th but}}

Montrons, par récurrence sur l'ordre du stabilisateur,
que pour tout corps de nombres~$k$, toute compactification lisse~$X$ d'un espace homogène~$V$
de~$\SL_n$ sur~$k$ à stabilisateur géométrique fini résoluble vérifie la propriété~$(E^+)$.
Au vu des \textsectiondouble\ref{subsec:reductionstabfini}--\ref{subsec:reductionstabfinires},
cela terminera la preuve du théorème~\ref{th but}.

Lorsque le stabilisateur est trivial, la variété~$X$ est rationnelle sur~$k$
puisque l'ensemble $H^1(k,\SL_n)$ est un singleton (théorème de Hilbert~90).
Dans ce cas, la propriété~$(E^+)$ pour~$X$ est donc équivalente, par
le lemme~\ref{lem:eplus invariant birationnel stable},
à la propriété~$(E^+)$ pour le point;
or cette dernière est bien vraie (voir \cite[Proposition~3.2.3]{liangarithmetic}
ou appliquer \cite[Theorem~8.3~(3), Lemma~8.2]{hw} à la fibration triviale $Z\times_k\P^1_k \to \P^1_k$).

Lorsque le stabilisateur n'est pas trivial, on applique
la proposition~\ref{prop:reductiontorseursuniversels}: quitte à effectuer une extension
des scalaires, il suffit d'établir~$(E^+)$ pour une compactification lisse
de chaque torseur universel de~$X$,
et même, vu le lemme~\ref{lem:eplus invariant birationnel stable}, pour une compactification lisse
d'un ouvert dense de chaque torseur universel de~$X$.
D'après le corollaire~\ref{cor:torseurs universels}
et le corollaire~\ref{cor:torseurs universels eh}, chaque torseur universel de~$X$
contient un ouvert dense admettant une compactification lisse munie d'un morphisme vers un espace projectif
dont la fibre générique est elle-même
une compactification lisse d'un
espace homogène de~$\SL_n$ à stabilisateur géométrique fini résoluble d'ordre strictement
inférieur à l'ordre du stabilisateur géométrique de~$V$.
L'hypothèse de récurrence et la proposition~\ref{prop:compatibilite eplus fibrations} permettent
de conclure la démonstration.

\begin{rmk}
Il est naturel de se demander si
une variante des arguments employés dans la preuve du théorème~\ref{th but}
permettrait d'établir l'énoncé suivant,
analogue, pour les zéro-cycles, du problème de Galois inverse:
si~$\Gamma$ est un groupe fini
et~$k$ un corps de nombres, il existe un entier $n\geq 1$ et des extensions finies $k_1,\dots,k_n$
de~$k$ tels que les degrés $[k_i:k]$ soient premiers entre eux dans leur ensemble
et que
pour chaque $i\in\{1,\dots,n\}$,
il existe une extension
finie galoisienne de~$k_i$ de groupe de Galois isomorphe à~$\Gamma$.  Cependant, cet énoncé est déjà connu,
même dans sa version \og{}régulière\fg{} (voir \cite[Remark~3, p.~367]{ctrcgalois}).
Au demeurant, il se démontre de façon élémentaire.
En effet, il n'est pas difficile de vérifier, par une récurrence sur la dimension,
que pour
tout sous-ensemble hilbertien~$H$
d'une variété~$X$ lisse et géométriquement irréductible sur un corps de nombres,
tout zéro-cycle sur~$X$ est rationnellement équivalent à un zéro-cycle supporté sur~$H$;
en particulier,
le sous-ensemble
hilbertien de $X=\SL_n/\Gamma$ associé au revêtement $\SL_n \to X$
contient un zéro-cycle de degré~$1$.
\end{rmk}

\bibliographystyle{monamsalpha}
\bibliography{zceh}
\typeout{get arXiv to do 4 passes: Label(s) may have changed. Rerun}
\end{document}